\documentclass{article}%
\usepackage{amsmath}
\usepackage{amsfonts}
\usepackage{amssymb}
\usepackage{graphicx}%
\newtheorem{theorem}{Theorem}

\newtheorem{corollary}[theorem]{Corollary}

\newtheorem{definition}[theorem]{Definition}

\newtheorem{lemma}[theorem]{Lemma}

\newtheorem{proposition}[theorem]{Proposition}
\newtheorem{remark}[theorem]{Remark}

\newenvironment{proof}[1][Proof]{\noindent\textbf{#1.} }{\ \rule{0.5em}{0.5em}}
\begin{document}

\title{Evaluation of Nonsymmetric Macdonald Superpolynomials at Special Points}
\author{Charles F. Dunkl\thanks{email: cfd5z@virginia.edu}\\Dept. of Mathematics\\University of Virginia\\Charlottesville VA 22904-4137 US}
\date{31 March 2021}
\maketitle

\begin{abstract}
In a preceding paper the theory of nonsymmetric Macdonald polynomials taking
values in modules of the Hecke algebra of type \$A\$ (Dunkl and Luque SLC
2012) was applied to such modules consisting of polynomials in anti-commuting
variables, to define nonsymmetric Macdonald superpolynomials. These
polynomials depend on two parameters \$%
%TCIMACRO{\TEXTsymbol{\backslash}}%
%BeginExpansion
$\backslash$%
%EndExpansion
left( q,t%
%TCIMACRO{\TEXTsymbol{\backslash}}%
%BeginExpansion
$\backslash$%
%EndExpansion
right) \$ and are defined by means of a Yang-Baxter graph. The present paper
determines the values of a subclass of the polynomials at the special points
\$%
%TCIMACRO{\TEXTsymbol{\backslash}}%
%BeginExpansion
$\backslash$%
%EndExpansion
left( 1,t,t\symbol{94}\{2\}\% ,%
%TCIMACRO{\TEXTsymbol{\backslash}}%
%BeginExpansion
$\backslash$%
%EndExpansion
ldots%
%TCIMACRO{\TEXTsymbol{\backslash}}%
%BeginExpansion
$\backslash$%
%EndExpansion
right) \$ or\$%
%TCIMACRO{\TEXTsymbol{\backslash}}%
%BeginExpansion
$\backslash$%
%EndExpansion
left( 1,t\symbol{94}\{-1\},t\symbol{94}\{-2\},%
%TCIMACRO{\TEXTsymbol{\backslash}}%
%BeginExpansion
$\backslash$%
%EndExpansion
ldots%
%TCIMACRO{\TEXTsymbol{\backslash}}%
%BeginExpansion
$\backslash$%
%EndExpansion
right) \$. The arguments use induction on the degree and computations with
products of generators of the Hecke algebra. The resulting formulas involve \$%
%TCIMACRO{\TEXTsymbol{\backslash}}%
%BeginExpansion
$\backslash$%
%EndExpansion
left( q,t%
%TCIMACRO{\TEXTsymbol{\backslash}}%
%BeginExpansion
$\backslash$%
%EndExpansion
right) \$-hookproducts. Evaluations are also found for Macdonald
superpolynomials having restricted symmetry andantisymmetry properties

\end{abstract}
\tableofcontents

\section{Introduction}

In the prequel \cite{D2020} of this paper we defined a representation of the
Hecke algebra of type $A$ on spaces of superpolynomials. By using the theory
of vector-valued nonsymmetric Macdonald polynomials developed by Luque and the
author \cite{DL2012} we constructed nonsymmetric Macdonald superpolynomials.
The basic theory including Cherednik operators, the Yang=Baxter graph method
for computing the Macdonald superpolynomials, and norm formulas were
described. The norm refers to an inner product with respect to which the
generators of the Hecke algebra are self-adjoint. The theory relies on
relating the Young tableaux approach to irreducible Hecke algebra modules to
polynomials in anti-commuting variables. Also that paper showed how to produce
symmetric and anti-symmetric Macdonald superpolynomials, and their norms, by
use of the technique of Baker and Forrester \cite{BF1999}. In the present
paper we consider the evaluation of the polynomials at certain special points.
The class of polynomials which lead to attractive formulas in pure product
form is relatively small. These values are expressed by shifted $q$%
-factorials, both ordinary (positive integer labeled) and the type labeled by
partitions, and $\left(  q,t\right)  $-hook products.

In Section \ref{SectBG} one finds the necessary background on the Hecke
algebra of type $A$ and its representations on polynomials in anti-commuting
(\textit{fermionic}) variables and on superpolynomials which combine commuting
(\textit{bosonic}) and anti-commuting variables. This section also defines the
Cherednik operators, a pairwise commuting set, whose simultaneous eigenvectors
are called nonsymmetric Macdonald superpolynomials. They are constructed
starting from degree zero by means of the Yang-Baxter graph. The necessary
details from \cite{D2020} are briefly given. Section \ref{SectEval} presents
the main results with proofs about the evaluations; there are two types with
similar arguments. The methods rely on steps in the graph to determine the
values starting from degree zero. Some of the arguments are fairly technical
computations using products of generators of the Hecke algebra. The definition
of $\left(  q,t\right)  $-hook products and their use in the evaluation
formulas are presented in Section \ref{SectHkPro}. The evaluations are
extended to Macdonald polynomials, of the types studied in the previous
sections, with restricted symmetry and antisymmetry properties in Section
\ref{SectHkPro}. The conclusion and ideas for further investigations in
Section \ref{SectConc} conclude the paper.

\section{\label{SectBG}Background}

\subsection{The Hecke algebra}

The Hecke algebra $\mathcal{H}_{N}\left(  t\right)  $ of type $A_{N-1}$ with
parameter $t$ is the associative algebra over an extension field of
$\mathbb{Q}$, generated by $\left\{  T_{1},\ldots,T_{N-1}\right\}  $ subject
to the braid relations%
\begin{subequations}
\begin{gather}
T_{i}T_{i+1}T_{i}=T_{i+1}T_{i}T_{i+1},~1\leq i<N-1,\label{Hbrd}\\
T_{i}T_{j}=T_{j}T_{i},~\left\vert i-j\right\vert \geq2,\nonumber
\end{gather}
and the quadratic relations%
\end{subequations}
\begin{equation}
\left(  T_{i}-t\right)  \left(  T_{i}+1\right)  =0,~1\leq i<N, \label{Hquad}%
\end{equation}
where $t$ is a generic parameter (this means $t^{n}\neq1$ for $2\leq n\leq N$,
and $t\neq0$). The quadratic relation implies $T_{i}^{-1}=\frac{1}{t}\left(
T_{i}+1-t\right)  $. There is a commutative set of \textit{Jucys-Murphy
elements} in $\mathcal{H}_{N}\left(  t\right)  $ defined by $\omega
_{N}=1,\omega_{i}=t^{-1}T_{i}\omega_{i+1}T_{i}$ for $1\leq i<N$, that is,%
\[
\omega_{i}=t^{i-N}T_{i}T_{i+1}\cdots T_{N-1}T_{N-1}T_{N-2}\cdots T_{i}%
\]
Simultaneous eigenvectors of $\left\{  \omega_{i}\right\}  $ form bases of
irreducible representations of the algebra. The symmetric group $\mathcal{S}%
_{N}$ is the group of permutations of $\left\{  1,2,\ldots,N\right\}  $ and is
generated by the simple reflections (adjacent transpositions) $\left\{
s_{i}:1\leq i<N\right\}  $, where $s_{i}$ interchanges $i,i+1$ and fixes the
other points (the $s_{i}$ satisfy the braid relations and $s_{i}^{2}=1$).

\subsection{Fermionic polynomials}

Consider polynomials in $N$ anti-commuting (fermionic) variables $\theta
_{1},\theta_{2},\ldots,\theta_{N}$. They satisfy $\theta_{i}^{2}=0$ and
$\theta_{i}\theta_{j}+\theta_{j}\theta_{i}=0$ for $i\neq j$. The basis for
these polynomials consists of monomials labeled by subsets of $\left\{
1,2,\ldots,N\right\}  $:%
\[
\phi_{E}:=\theta_{i_{1}}\cdots\theta_{i_{m}},~E=\left\{  i_{1},i_{2}%
,\cdots,i_{m}\right\}  ,1\leq i_{1}<i_{2}<\cdots<i_{m}\leq N.
\]

The polynomials have coefficients in an extension field of $\mathbb{Q}\left(
q,t\right)  $ with transcendental $q,t$, or generic $q,t$ satisfying
$q,t\neq0$ $q^{a}\neq1,q^{a}t^{n}\neq1$ for $a\in\mathbb{Z}$ and
$n\neq2,3,\ldots,N$.

\begin{definition}
$\mathcal{P}:=\mathrm{span}\left\{  \phi_{E}:E\subset\left\{  1,\ldots
,N\right\}  \right\}  $ and $\mathcal{P}_{m}:=\mathrm{span}\left\{  \phi
_{E}:\#E=m\right\}  $ for $0\leq m\leq N$. The fermionic degree of $\phi_{E}$
is $\#E$.
\end{definition}

This is a brief description of the action of $T_{i}$ on $\mathcal{P}$: suppose
$j\in E_{1}$ implies $j<i$, and $j\in E_{2}$ implies $j>i+1$: then%
\begin{align}
T_{i}\phi_{E_{1}}\theta_{i}\theta_{i+1}\phi_{E_{2}} &  =-\phi_{E_{1}}%
\theta_{i}\theta_{i+1}\phi_{E_{2}},T_{i}\phi_{E_{1}}\phi_{E_{2}}=t\phi_{E_{1}%
}\phi_{E_{2}},\label{defTth}\\
T_{i}\phi_{E_{1}}\theta_{i}\phi_{E_{2}} &  =\phi_{E_{1}}\theta_{i+1}%
\phi_{E_{2}},T_{i}\phi_{E_{1}}\theta_{i+1}\phi_{E_{2}}=T_{i}\phi_{E_{1}%
}\left(  t\theta_{i}+\left(  t-1\right)  \theta_{i+1}\right)  \phi_{E_{2}%
}.\nonumber
\end{align}
Then $\left\{  T_{i}:1\leq i<N\right\}  $ satisfy the braid and quadratic relations.

There are two degree-changing linear maps which commute with the Hecke algebra action.

\begin{definition}
\label{defMD}\{\{For $n\in\mathbb{Z}$ set $\sigma\left(  n\right)  :=\left(
-1\right)  ^{n}$ and\}\} for $E\subset\left\{  1,2,\ldots,N\right\}  $, $1\leq
i\leq N$ set $s\left(  i,E\right)  :=\#\left\{  j\in E:j<i\right\}  $. Define
the operators $\partial_{i}$ and $\widehat{\theta}_{i}$ by $\partial_{i}%
\theta_{i}\phi_{E}=\phi_{E},\partial_{i}\phi_{E}=0$ and $\widehat{\theta}%
_{i}\phi_{E}=\theta_{i}\phi_{E}=\left(  -1\right)  ^{s\left(  i,E\right)
}\phi_{E\cup\left\{  i\right\}  }$ for $i\notin E$, while $\widehat{\theta
}_{i}\phi_{E}=0$ for $i\in E$ (also $i\in E$ implies $\phi_{E}=\left(
-1\right)  ^{s\left(  i,E\right)  }\theta_{i}\phi_{E\backslash\left\{
i\right\}  }$ and $\partial_{i}\phi_{E}=\left(  -1\right)  ^{s\left(
i,E\right)  }\phi_{E\backslash\left\{  i\right\}  }$). Define $M:=\sum
_{i=1}^{N}\widehat{\theta}_{i}$ and $D:=\sum_{i=1}^{N}t^{i-1}\partial_{i}$.
\end{definition}

It is clear that $D^{2}=0=M^{2}$. For $n=0,1,2,\ldots$ let $\left[  n\right]
_{t}:=\dfrac{1-t^{n}}{1-t}$.

\begin{proposition}
\label{commMD}$M$ and $D$ commute with $T_{i}$ for $1\leq i<N,$and
$MD+DM=\left[  N\right]  _{t}.$
\end{proposition}

The spaces $\mathcal{P}_{m,0}:=\ker D\cap\mathcal{P}_{m}$ and $\mathcal{P}%
_{m+1,1}:=\ker M\cap\mathcal{P}_{m+1}$ are irreducible $\mathcal{H}_{N}\left(
t\right)  $-modules and are isomorphic under the map $D:\mathcal{P}%
_{m+1,1}\rightarrow\mathcal{P}_{m,0}$ and are of isotype $\left(
N-m,1^{m}\right)  $.The representations of $\mathcal{H}_{N}\left(  t\right)  $
occurring in this paper correspond to reverse standard Young tableaus (RSYT)
of hook shape (see Dipper and James \cite{DJ1986} for details of the
representation theory) These are labeled by partitions $\left(  N-n,1^{n}%
\right)  $ of $N$ and are graphically described by Ferrers diagrams: boxes at
$\left\{  \left[  1,i\right]  :1\leq i\leq N-n\right\}  \cup\left\{  \left[
j,1\right]  :2\leq j\leq n\right\}  $. The numbers $\left\{  1,2,\ldots
,N\right\}  $ are entered in the boxes in decreasing order in the row and in
the column. For a given RSYT $Y$ let $Y\left[  a,b\right]  $ be the entry at
$\left[  a,b\right]  $ and define the content $c\left(  Y\left[  a,b\right]
,Y\right)  :=b-a$. The vector $\left[  c\left(  i,Y\right)  :1\leq i\leq
N\right]  $ is called the content vector of $Y$. It defines $Y$ uniquely
(trivially true for hook tableaux). The representation of $\mathcal{H}%
_{N}\left(  t\right)  $ is defined on the span of the RSYT's of shape $\left(
N-n,1^{n}\right)  $ in such a way that $\omega_{i}Y=t^{c\left(  i,Y\right)
}Y$ for $1\leq i\leq N$. We use a space-saving way of displaying an RSYT in
two rows, with the second row consisting of the entries $Y_{E}\left[
2,1\right]  ,Y_{E}\left[  3,1\right]  ,\ldots$.. Note that $Y\left[
1,1\right]  =N$ always.

As example let $N=8,n=3,$%
\begin{equation}
Y=%
\begin{bmatrix}
8 & 6 & 4 & 3 & 1\\
\cdot & 7 & 5 & 2 &
\end{bmatrix}
\label{exYE}%
\end{equation}
and $\left[  c\left(  i,Y\right)  \right]  _{i=1}^{8}=\left[
4,-3,3,2,-2,1,-1,0\right]  $.

We showed \cite{D2020} that $\mathcal{P}_{m}$ is a direct sum of the
$\mathcal{H}_{N}\left(  t\right)  $-modules corresponding to $\left(
N-m,1^{m}\right)  $ and $\left(  N+1-m,1^{m-1}\right)  $ ; $\ker
D\cap\mathcal{P}_{m}$ and $\ker M\cap\mathcal{P}_{m}$ respectively.

\subsection{The module $\ker D\cap\mathcal{P}_{m}$}

The basis of $\ker D\cap\mathcal{P}_{m}$ is described as follows: Let
$\mathcal{Y}_{0}:=\left\{  E:\#E=m+1,N\in E\right\}  $ and for $E\in
\mathcal{Y}_{0}$ let $\psi_{E}=D\phi_{E}$. Associate $E$ to the RSYT $Y_{E}$
which contains the elements of $E$ in decreasing order in column $1$, that is,
$\left\{  \left[  j,1\right]  :1\leq j\leq m+1\right\}  $, and the elements of
$E^{C}$ in $\left\{  \left[  1,i\right]  :2\leq i\leq N-m\right\}  $. In the
example $Y=Y_{R}$ with $E=\left\{  2,5,7,8\right\}  $. The content vector of
$E$ is defined by $c\left(  i,E\right)  =c\left(  i,Y_{E}\right)  $. For each
$E\in\mathcal{Y}_{0}$ there is a polynomial $\tau_{E}\in\ker D\cap
\mathcal{P}_{m}$ such that $\omega_{i}\tau_{E}=t^{c\left(  i,E\right)  }%
\tau_{E}$ for $1\leq i\leq N$, and if $\mathrm{inv}\left(  E\right)  =k$ then
$\tau_{E}-t^{k}\psi_{E}\in\mathrm{span}\left\{  \psi_{E^{\prime}}%
:\mathrm{inv}\left(  E^{\prime}\right)  <k\right\}  $. In particular if
$E=\left\{  N-m,N-m+1,\ldots,N\right\}  $ then $\tau_{E}=D\phi_{E}$ (and this
is one of the two cases that are used here). For example suppose $N=7,m=3$
then%
\[
Y_{E}=%
\begin{bmatrix}
7 & 3 & 2 & 1\\
\cdot & 6 & 5 & 4
\end{bmatrix}
,
\]
$\left[  c\left(  i,E\right)  \right]  _{i=1}^{7}=\left[
3,2,1,-3,-2,-1,0\right]  $, and $\tau_{E}=D\left(  \theta_{4}\theta_{5}%
\theta_{6}\theta_{7}\right)  =t^{3}\theta_{5}\theta_{6}\theta_{7}-t^{4}%
\theta_{4}\theta_{6}\theta_{7}+t^{5}\theta_{4}\theta_{5}\theta_{7}-t^{6}%
\theta_{4}\theta_{5}\theta_{6}$.

\subsection{The module $\ker M\cap\mathcal{P}_{m}$}

The basis of $\ker M\cap\mathcal{P}_{m}$ is described as follows: Let
$\mathcal{Y}_{1}:=\left\{  F:\#F=m-1,N\notin F\right\}  $ and for
$F\in\mathcal{Y}_{1}$ let $\eta_{F}=M\phi_{F}$. Associate $F$ to the RSYT
$Y_{F}$ which contains the elements of $F$ in decreasing order in column $1$,
that is, $\left\{  \left[  j,1\right]  :2\leq j\leq m\right\}  $, and the
elements of $F^{C}$ in $\left\{  \left[  1,i\right]  :1\leq i\leq
N-m+11\right\}  $. In the example (\ref{exYE}) $Y=Y_{F}$ with $F=\left\{
2,5,7\right\}  $. As before the content vector of $F$ is defined by $c\left(
i,F\right)  =c\left(  i,Y_{F}\right)  $. For each $F\in\mathcal{Y}_{1}$ there
is a polynomial $\tau_{F}\in\ker M\cap\mathcal{P}_{m}$ such that $\omega
_{i}\tau_{F}=t^{c\left(  i,F\right)  }\tau_{F}$ for $1\leq i\leq N$, and if
$\mathrm{inv}\left(  F\right)  =k$ then $\tau_{F}-\eta_{F}\in\mathrm{span}%
\left\{  \psi_{F^{\prime}}:\mathrm{inv}\left(  F^{\prime}\right)  >k\right\}
$. Note that $F\in\mathcal{Y}_{1}$ implies $0\leq\mathrm{inv}\left(  F\right)
\leq\left(  m-1\right)  \left(  N-m+1\right)  $ and the maximum value occurs
at $F=\left\{  1,2,\ldots,m-1\right\}  $. This case is the second of those to
be studied here. For this set $\tau_{F}=M\phi_{F}$. As example let $N=7,m=5$
then%
\[
Y_{F}=%
\begin{bmatrix}
7 & 6 & 5 &  & \\
\cdot & 4 & 3 & 2 & 1
\end{bmatrix}
,
\]
$\left[  c\left(  i,F\right)  \right]  _{i=1}^{7}=\left[
-4,-3,-2,-1,2,1,0\right]  $, and $\tau_{F}=\theta_{1}\theta_{2}\theta
_{3}\theta_{4}\left(  \theta_{5}+\theta_{6}+\theta_{7}\right)  $.

\subsection{Superpolynomials}

We extend the polynomials in $\left\{  \theta_{i}\right\}  $ by adjoining $N$
commuting variables $x_{1},\ldots,x_{N}$ (that is $\left[  x_{i},x_{j}\right]
=0,\left[  x_{i},\theta_{j}\right]  =0,\theta_{i}\theta_{j}=-\theta_{j}%
\theta_{i}$ for all $i,j$). Each polynomial is a sum of monomials $x^{\alpha
}\phi_{E}$ where $E\subset\left\{  1,2,\ldots,N\right\}  $ and $\alpha
\in\mathbb{N}_{0}^{N},x^{\alpha}:=\prod\limits_{i=1}^{N}x_{i}^{\alpha_{i}}$.
The partitions in $\mathbb{N}_{0}^{N}$ are denoted by $\mathbb{N}_{0}^{N,+}$
($\lambda\in\mathbb{N}_{0}^{N,+}$ if and only if $\lambda_{1}\geq\lambda
_{2}\geq\ldots\geq\lambda_{N}$). The \textit{fermionic} degree of this
monomial is $\#E$ and the \textit{bosonic} degree is $\left\vert
\alpha\right\vert :=\sum_{i=1}^{N}\alpha_{i}$. The symmetric group
$\mathcal{S}_{N}$ acts on the variables by $\left(  xw\right)  _{i}%
=x_{w\left(  i\right)  }$ and on exponents by $\left(  w\alpha\right)
_{i}=\alpha_{w^{-1}\left(  i\right)  }$ for $1\leq i\leq N,w\in\mathcal{S}%
_{N}$ (consider $x$ as a row vector, $\alpha$ as a column vector and $w$ as a
permutation matrix, $\widetilde{w}_{i,j}=\delta_{i,w\left(  j\right)  }$, then
$xw=x\widetilde{w}$ and $w\alpha=\widetilde{w}\alpha$). Thus $\left(
xw\right)  ^{\alpha}=x^{w\alpha}$. Let $s\mathcal{P}_{m}:=\mathrm{span}%
\left\{  x^{\alpha}\phi_{E}:\alpha\in\mathbb{N}_{0}^{N},\#E=m\right\}  $. Then
using the decomposition $\mathcal{P}_{m}=\mathcal{P}_{m,0}\oplus
\mathcal{P}_{m,1}$ let%
\begin{align*}
s\mathcal{P}_{m,0}  &  =\mathrm{span}\left\{  x^{\alpha}\psi_{E}:\alpha
\in\mathbb{N}_{0}^{N},E\in\mathcal{Y}_{0}\right\}  ,\\
s\mathcal{P}_{m,1}  &  =\mathrm{span}\left\{  x^{\alpha}\eta_{E}:\alpha
\in\mathbb{N}_{0}^{N},E\in\mathcal{Y}_{1}\right\}  .
\end{align*}
The Hecke algebra $\mathcal{H}_{N}\left(  t\right)  $ is represented on
$s\mathcal{P}_{m}$. This allows us to apply the theory of nonsymmetric
Macdonald polynomials taking values in $\mathcal{H}_{N}\left(  t\right)
$-modules (see \cite{DL2012} ).

\begin{definition}
Suppose $p\in s\mathcal{P}_{m}$ and $1\leq i<N$ then set%
\begin{equation}
\boldsymbol{T}_{i}p\left(  x;\theta\right)  :=\left(  1-t\right)  x_{i+1}%
\frac{p\left(  x;\theta\right)  -p\left(  xs_{i};\theta\right)  }%
{x_{i}-x_{i+1}}+T_{i}p\left(  xs_{i};\theta\right)  . \label{defTTp}%
\end{equation}

\end{definition}

Note that $T_{i}$ acts on the $\theta$ variables according to Formula
(\ref{defTth}).

\begin{definition}
Let $T^{\left(  N\right)  }=T_{N-1}T_{N-2}\cdots T_{1}$ and for $p\in
s\mathcal{P}_{m}$ and $1\leq i\leq N$%
\begin{align*}
\boldsymbol{w}p\left(  x;\theta\right)   &  :=T^{\left(  N\right)  }p\left(
qx_{N},x_{1},x_{2},\ldots,x_{N-1};\theta\right)  ,\\
\xi_{i}p\left(  x;\theta\right)   &  :=t^{i-N}\boldsymbol{T}_{i}%
\boldsymbol{T}_{i+1}\cdots\boldsymbol{T}_{N-1}\boldsymbol{wT}_{1}%
^{-1}\boldsymbol{T}_{2}^{-1}\cdots\boldsymbol{T}_{i-1}^{-1}p\left(
x;\theta\right)  .
\end{align*}

\end{definition}

The operators $\xi_{i}$ are Cherednik operators, defined by Baker and
Forrester \cite{BF1997} (see Braverman et al \cite{BDLM2012} for the
significance of these operators in double affine Hecke algebras). They
mutually commute (the proof in the vector-valued situation is in \cite[Thm.
3.8]{DL2012}). The simultaneous eigenfunctions are called nonsymmetric
Macdonald polynomials. They have a triangularity property with respect to the
partial order $\rhd$ on the compositions $\mathbb{N}_{0}^{N}$, which is
derived from the dominance order:
\begin{align*}
\alpha &  \prec\beta~\Longleftrightarrow\sum_{j=1}^{i}\alpha_{j}\leq\sum
_{j=1}^{i}\beta_{j},~1\leq i\leq N,~\alpha\neq\beta\text{,}\\
\alpha\lhd\beta &  \Longleftrightarrow\left(  \left\vert \alpha\right\vert
=\left\vert \beta\right\vert \right)  \wedge\left[  \left(  \alpha^{+}%
\prec\beta^{+}\right)  \vee\left(  \alpha^{+}=\beta^{+}\wedge\alpha\prec
\beta\right)  \right]  \text{.}%
\end{align*}
The rank function on compositions is involved in the formula for an NSMP.

\begin{definition}
For $\alpha\in\mathbb{N}_{0}^{N},1\leq i\leq N$%
\[
r_{\alpha}\left(  i\right)  :=\#\left\{  j:\alpha_{j}>\alpha_{i}\right\}
+\#\left\{  j:1\leq j\leq i,\alpha_{j}=\alpha_{i}\right\}  \text{,}%
\]
then $r_{\alpha}\in\mathcal{S}_{N}\text{. There is a shortest expression
}r_{\alpha}=s_{i_{1}}s_{i_{2}}\ldots s_{i_{k}}$ and $R_{\alpha}:=\left(
T_{i_{1}}T_{i_{2}}\cdots T_{i_{k}}\right)  ^{-1}\in\mathcal{H}_{N}\left(
t\right)  $ (that is, $R_{\alpha}=T\left(  r_{\alpha}\right)  ^{-1}$).
\end{definition}

A consequence is that $r_{\alpha}\alpha=\alpha^{+}$, the \textit{nonincreasing
rearrangement} of $\alpha$, for any $\alpha\in\mathbb{N}_{0}^{N}$ , and
$r_{\alpha}=I$ if and only if $\alpha\in\mathbb{N}_{0}^{N,+}$.

\begin{theorem}
\label{Mform}(\cite[Thm. 4.12]{DL2012}) Suppose $\alpha\in\mathbb{N}_{0}^{N}$
and $E\in\mathcal{Y}_{k}$, $k=0,1$ then there exists a $\left(  \xi
_{i}\right)  $-simultaneous eigenfunction%
\begin{equation}
M_{\alpha,E}\left(  x;\theta\right)  =t^{e\left(  \alpha^{+},E\right)
}q^{\beta\left(  \alpha\right)  }x^{\alpha}R_{\alpha}\left(  \tau_{E}\left(
\theta\right)  \right)  +\sum_{\beta\vartriangleleft\alpha}x^{\beta}%
v_{\alpha,\beta,E}\left(  \theta;q,t\right)  \label{MaEeqn}%
\end{equation}
where $v_{\alpha,\beta,E}\left(  \theta;q,t\right)  \in\mathcal{P}_{m,k}$ and
its coefficients are rational functions of $q,t$. Also $\xi_{i}M_{\alpha
,E}\left(  x;\theta\right)  =\zeta_{\alpha,E}\left(  i\right)  M_{\alpha
,E}\left(  x;\theta\right)  $ where $\zeta_{\alpha,E}\left(  i\right)
=q^{\alpha_{i}}t^{c\left(  r_{\alpha}\left(  i\right)  ,E\right)  }$ for
$1\leq i\leq N.$ The exponents $\beta\left(  \alpha\right)  :=\sum_{i=1}%
^{N}\binom{\alpha_{i}}{2}$ and $e\left(  \alpha^{+},E\right)  :=\sum_{i=1}%
^{N}\alpha_{i}^{+}\left(  N-i+c\left(  i,E\right)  \right)  $.
\end{theorem}

The applications in the present paper require formulas for the transformation
(called a \textit{step}) $M_{\alpha,E}\rightarrow M_{s_{i}\alpha,E}$ when
$\alpha_{i+1}>\alpha_{i}$:
\begin{equation}
M_{s_{i}\alpha,E}\left(  x;\theta\right)  =\left(  \boldsymbol{T}_{i}%
+\frac{1-t}{1-\zeta_{\alpha,E}\left(  i+1\right)  /\zeta_{\alpha,E}\left(
i\right)  }\right)  M_{\alpha,E}\left(  x;\theta\right)  . \label{step}%
\end{equation}
and for the \textit{affine step}:
\begin{align*}
\Phi\alpha &  =\left(  \alpha_{2},\alpha_{3},\ldots,\alpha_{N},\alpha
_{1}+1\right) \\
\zeta_{\Phi\alpha,E}  &  =\left[  \zeta_{\alpha,E}\left(  2\right)
,\zeta_{\alpha,E}\left(  3\right)  ,\ldots,\zeta_{\alpha,E}\left(  N\right)
,q\zeta_{\alpha,E}\left(  1\right)  \right] \\
M_{\Phi\alpha,E}\left(  x\right)   &  =x_{N}\boldsymbol{w}M_{\alpha,E}\left(
x\right)  .
\end{align*}
Two other key relations are $\zeta_{\alpha,E}\left(  i+1\right)
=t\zeta_{\alpha,E}\left(  i\right)  $ implies $\left(  \boldsymbol{T}%
_{i}+1\right)  M_{\alpha,E}=0$ and $\zeta_{\alpha,E}\left(  i+1\right)
=t^{-1}\zeta_{\alpha,E}\left(  i\right)  $ implies $\left(  \boldsymbol{T}%
_{i}-t\right)  M_{\alpha,E}=0$.

\section{\label{SectEval}Evaluations and Steps}

We consider two types of evaluations: (0) $x^{\left(  0\right)  }=\left(
1,t,t^{2},\ldots,t^{N-1}\right)  $, $E=\left\{  N-m,N-m+1,\ldots,N\right\}  ,$
$\alpha\in\mathbb{N}_{0}^{N}$ with $\alpha_{i}=0$ for $N-m\leq i\leq N$ , and
$M_{\alpha,E}\left(  x^{\left(  0\right)  }\right)  =V^{\left(  0\right)
}\left(  \alpha\right)  \tau_{E}$; (1) $x^{\left(  1\right)  }=\left(
1,t^{-1},t^{-2},\ldots,t^{1-N}\right)  $, $F=\left\{  1,2,\ldots,m\right\}  $,
$\alpha\in\mathbb{N}_{0}^{N}$ with $\alpha_{i}=0$ for $m+1\leq i\leq N$, and
$M_{\alpha,F}\left(  x^{\left(  1\right)  }\right)  =V^{\left(  1\right)
}\left(  \alpha\right)  \tau_{F}$.

\begin{definition}
Let $\mathcal{N}_{0}:=\left\{  \alpha\in\mathbb{N}_{0}^{N}:i\geq
N-m\Longrightarrow\alpha_{i}=0\right\}  $, $\mathcal{N}_{0}^{+}:=\mathcal{N}%
_{0}\cap\mathbb{N}_{0}^{N,+}$. Let $\mathcal{N}_{1}:=\left\{  \alpha
\in\mathbb{N}_{0}^{N}:i>m\Longrightarrow\alpha_{i}=0\right\}  $,
$\mathcal{N}_{1}^{+}:=\mathcal{N}_{1}\cap\mathbb{N}_{0}^{N,+}$.
\end{definition}

Conceptually the two derivations are very much alike, but there are
differences involving signs and powers of $t$ that need careful attention. We
begin by expressing $V^{\left(  0\right)  }\left(  \alpha\right)  $ and
$V^{\left(  1\right)  }\left(  \alpha\right)  $ in terms of $V^{\left(
0\right)  }\left(  \alpha^{+}\right)  $ and $V^{\left(  1\right)  }\left(
\alpha^{+}\right)  $. Since we are concerned with evaluations the following is
used throughout:

\begin{definition}
For a fixed point $x\in\mathbb{R}^{N}$ and $1\leq i<N$ let $b\left(
x;i\right)  =\dfrac{t-1}{1-x_{i}/x_{i+1}}$ $\left(  x_{i}\neq x_{i+1}\right)
$. In particular if $x_{i+1}=t^{n}x_{i}$ then let $\kappa_{n}:=b\left(
x;i\right)  $ for $n\in\mathbb{Z}\backslash\left\{  0\right\}  $. If $n\geq1$
then $\kappa_{n}=\dfrac{t^{n}}{\left[  n\right]  _{t}}$ and $\kappa
_{-n}=-\dfrac{1}{\left[  n\right]  _{t}}$.
\end{definition}

In terms of $b$ the evaluation formula for $\boldsymbol{T}_{i}$ is%
\begin{equation}
\boldsymbol{T}_{i}p\left(  x;\theta\right)  =b\left(  x;i\right)  p\left(
x;\theta\right)  +\left(  T_{i}-b\left(  x;i\right)  \right)  p\left(
xs_{i};\theta\right)  \label{Tbp}%
\end{equation}

The following are used repeatedly in the sequel.

\begin{lemma}
\label{Ti+1}Suppose for some $i<N$ there is a polynomial $p\left(
x;\theta\right)  $ and a point $y$ such that $\left(  \boldsymbol{T}%
_{i}+1\right)  p=0$ and $y_{i+1}=ty_{i}$ then $\left(  T_{i}+1\right)
p\left(  y;\theta\right)  =0$.
\end{lemma}

\begin{proof}
By hypothesis $b\left(  y;i\right)  =t$ and thus $\left(  1+t\right)  p\left(
y;\theta\right)  +\left(  T_{i}-t\right)  p\left(  ys_{i};\theta\right)  =0$.
Then $\left(  1+t\right)  \left(  T_{i}+1\right)  p\left(  y;\theta\right)
=-\left(  T_{i}+1\right)  \left(  T_{i}-t\right)  p\left(  ys_{i}%
;\theta\right)  =0$.
\end{proof}

\begin{lemma}
\label{Ti-t}Suppose for some $i<N$ there is a polynomial $p\left(
x;\theta\right)  $ and a point $y$ such that $\left(  \boldsymbol{T}%
_{i}-t\right)  p=0$ and $y_{i}=ty_{i+1}$ then $\left(  T_{i}-t\right)
p\left(  y;\theta\right)  =0$.
\end{lemma}

\begin{proof}
By hypothesis $b\left(  y;i\right)  =-1$ and thus $\left(  -t-1\right)
p\left(  y;\theta\right)  +\left(  T_{i}+1\right)  p\left(  ys_{i}%
;\theta\right)  =0$. Then $\left(  1+t\right)  \left(  T_{i}-t\right)
p\left(  y;\theta\right)  =-\left(  T_{i}-t\right)  \left(  T_{i}+1\right)
p\left(  ys_{i};\theta\right)  =0$.
\end{proof}

In type (0) $\zeta_{\alpha,E}\left(  i\right)  =t^{i-N}$ for $N-m\leq i\leq N$
which implies $\left(  \boldsymbol{T}_{i}+1\right)  M_{\alpha,E}=0$ for
$N-m\leq i<N$.

\begin{lemma}
\label{M0ctauE}Suppose $M_{\alpha,E}$ is of type (0) and $x_{i+1}=tx_{i}$ for
$N-m\leq i<N$ then $M_{\alpha,E}\left(  x\right)  =c\tau_{E}$ for some
constant depending on $x$, and $\left(  T_{i}-t\right)  M_{\alpha,E}\left(
x\right)  =0$ for $1\leq i<N-m-1$.
\end{lemma}

\begin{proof}
From $\left(  \boldsymbol{T}_{i}+1\right)  M_{\alpha,E}=0$ and Lemma
\ref{Ti+1} it follows that $\left(  T_{i}+1\right)  M_{\alpha,E}\left(
x\right)  =0$ for $N-m\leq i<N$. Thus $\left(  \omega_{i}-t^{i-N}\right)
M_{\alpha,E}\left(  x;\theta\right)  =0$ for $N-m\leq i\leq N$, and this
implies $M_{\alpha,E}\left(  x;\theta\right)  $ is a multiple of $\tau_{E}$
(the contents $\left[  c\left(  i,E^{\prime}\right)  \right]  _{i=N-m}^{N}$
determine $E^{\prime}$ uniquely). Furthermore $\left(  T_{i}-t\right)
\tau_{E}=0$ for $1\leq i<N-m-1$ (since $1,2,\ldots,N-m-1$ are in the same row
of $Y_{E}$).
\end{proof}

\begin{proposition}
Suppose $\alpha\in\mathcal{N}_{0}$ and $\alpha_{i}<\alpha_{i+1}$ (implying
$i+1<N-m)$ and $z=\zeta_{\alpha,E}\left(  i+1\right)  /\zeta_{\alpha,E}\left(
i\right)  $ then%
\[
M_{s_{i}\alpha,E}\left(  x^{\left(  0\right)  };\theta\right)  =\frac
{1-tz}{1-z}M_{\alpha,E}\left(  x^{\left(  0\right)  };\theta\right)  .
\]

\end{proposition}

\begin{proof}
From (\ref{step}) and (\ref{Tbp}) with $b\left(  x^{\left(  0\right)
},i\right)  =t$ it follows that%
\begin{align*}
M_{s_{i}\alpha,E}\left(  x^{\left(  0\right)  };\theta\right)   &  =\left(
t+\frac{1-t}{1-z}\right)  M_{\alpha,E}\left(  x^{\left(  0\right)  }%
;\theta\right)  +\left(  T_{i}-t\right)  M_{\alpha,E}\left(  x^{\left(
0\right)  }s_{i};\theta\right) \\
&  =\frac{1-tz}{1-z}M_{\alpha,E}\left(  x^{\left(  0\right)  };\theta\right)
,
\end{align*}
because $x^{\left(  0\right)  }s_{i}$satisfies the hypotheses of the Lemma
implying $\left(  T_{i}-t\right)  M_{\alpha,E}\left(  x^{\left(  0\right)
}s_{i};\theta\right)  =0$.
\end{proof}

The following products are used to relate $V^{\left(  k\right)  }\left(
\alpha\right)  $ to $V^{\left(  k\right)  }\left(  \alpha^{+}\right)  ,~k=0,1$.

\begin{definition}
Let $u_{0}\left(  z\right)  :=\dfrac{t-z}{1-z},u_{1}\left(  z\right)
:=\dfrac{1-tz}{1-z}$ . Suppose $\beta\in\mathbb{N}_{0}^{N}$ and $E^{\prime}%
\in\mathcal{Y}^{0}\cup\mathcal{Y}^{1}$ and $k=0,1$ then
\[
\mathcal{R}_{k}\left(  \beta,E^{\prime}\right)  :=\prod\limits_{1\leq i<j\leq
N,~\beta_{i}<\beta_{j}}u_{k}\left(  q^{\beta_{j}-\beta_{i}}t^{c\left(
r_{\beta}\left(  j\right)  ,E^{\prime}\right)  -c\left(  r_{\beta}\left(
i\right)  ,E^{\prime}\right)  }\right)  .
\]

\end{definition}

Note that the argument of $u_{k}$ is $\zeta_{\beta,E^{\prime\prime}}\left(
j\right)  /\zeta_{\beta,E^{\prime\prime}}\left(  i\right)  $ and there are
$\mathrm{inv}\left(  \beta\right)  $ factors, where%
\[
\mathrm{inv}\left(  \beta\right)  :=\left\{  \left(  i,j\right)  :1\leq
i<j\leq N,\beta_{i}<\beta_{j}\right\}  .
\]

\begin{lemma}
\label{RbRb1}If $\beta_{i}<\beta_{i+1}$ then $\mathcal{R}_{k}\left(
\beta,E^{\prime}\right)  =u_{k}\left(  \zeta_{\beta,E^{\prime\prime}}\left(
i+1\right)  /\zeta_{\beta,E^{\prime\prime}}\left(  i\right)  \right)
\mathcal{R}_{k}\left(  s_{i}\beta,E^{\prime}\right)  $.
\end{lemma}

\begin{proof}
The only factor that appears in $\mathcal{R}_{k}\left(  \beta,E^{\prime
}\right)  $ but not in $\mathcal{R}_{k}\left(  s_{i}\beta,E^{\prime}\right)  $
is $u_{k}\left(  \zeta_{\beta,E^{\prime\prime}}\left(  i+1\right)
/\zeta_{\beta,E^{\prime\prime}}\left(  i\right)  \right)  $.
\end{proof}

For the special case type (0) $1\leq r_{\alpha}\left(  i\right)  <N-m$ we find
$c\left(  r_{\alpha}\left(  i\right)  ,E\right)  =N-m-r_{\alpha}\left(
i\right)  $ and%
\[
\mathcal{R}_{1}\left(  \alpha,E\right)  =\prod\limits_{1\leq i<j<N-m,~\alpha
_{i}<\alpha_{j}}u_{1}\left(  q^{\alpha_{j}-\alpha_{i}}t^{r_{\alpha}\left(
i\right)  -r_{\alpha}\left(  j\right)  }\right)  .
\]

\begin{proposition}
\label{V0alpha}Suppose $\alpha\in\mathcal{N}_{0}$ then $M_{\alpha,E}\left(
x^{\left(  0\right)  };\theta\right)  =V^{\left(  0\right)  }\left(
\alpha\right)  \tau_{E}$ and%
\[
V^{\left(  0\right)  }\left(  \alpha\right)  =\mathcal{R}_{1}\left(
\alpha,E\right)  ^{-1}V^{\left(  0\right)  }\left(  \alpha^{+}\right)  .
\]

\end{proposition}

\begin{proof}
By Lemma \ref{M0ctauE} $M_{\alpha,E}\left(  x^{\left(  0\right)  }%
;\theta\right)  $ is a multiple of $\tau_{E}$. For the product formula argue
by induction on $\mathrm{inv}\left(  \alpha\right)  $. If $\lambda
\in\mathbb{N}_{0}^{N,+}$ then $\mathcal{R}_{1}\left(  \lambda,E\right)  =1$.
If $\alpha_{i}<\alpha_{i+1}$ then%
\[
\frac{V^{\left(  0\right)  }\left(  s_{i}\alpha\right)  }{V^{\left(  0\right)
}\left(  \alpha\right)  }=u_{1}\left(  \frac{\zeta_{\alpha,E}\left(
i+1\right)  }{\zeta_{\alpha,E}\left(  i\right)  }\right)  =\frac
{\mathcal{R}_{1}\left(  \alpha,E\right)  }{\mathcal{R}_{1}\left(  s_{i}%
\alpha,E\right)  }.
\]

\end{proof}

In type (1) $\zeta_{\alpha}\left(  i\right)  =t^{N-i}$ for $m+1\leq i\leq N$
which implies $\left(  \boldsymbol{T}_{i}-t\right)  M_{\alpha,E}=0$ for
$m+1\leq i<N$.

\begin{lemma}
\label{M1ctauF}Suppose $M_{\alpha,F}$ is of type (1) and $x_{i}=tx_{i+1}$ for
$m+1\leq i<N$ then $M_{\alpha,E}\left(  x\right)  =c\tau_{F}$ for some
constant depending on $x$, and $\left(  T_{i}+1\right)  M_{\alpha,F}\left(
x\right)  =0$ for $1\leq i<m$.
\end{lemma}

\begin{proof}
From $\left(  \boldsymbol{T}_{i}-t\right)  M_{\alpha,F}=0$ and Lemma
\ref{Ti-t} it follows that $\left(  T_{i}-t\right)  M_{\alpha,F}\left(
x;\theta\right)  =0$ for $m+1\leq i<N$. Thus $\left(  \omega_{i}%
-t^{N-i}\right)  M_{\alpha,F}\left(  x;\theta\right)  =0$ for $m+1\leq i\leq
N$, and this implies $M_{\alpha,F}\left(  x;\theta\right)  $ is a multiple of
$\tau_{F}$ (the contents $\left[  c\left(  i,E^{\prime}\right)  \right]
_{i=m+1}^{N}$ determine $E^{\prime}$ uniquely). Thus $\left(  T_{i}+1\right)
\tau_{F}=0$ for $1\leq i<m$ (since $1,2,\ldots,m$ are in the same column of
$Y_{F}$).
\end{proof}

\begin{proposition}
Suppose $\alpha\in\mathcal{N}_{1}$ and $\alpha_{i}<\alpha_{i+1}$ (so that
$i+1\leq m)$ and $z=\zeta_{\alpha,F}\left(  i+1\right)  /\zeta_{\alpha
,F}\left(  i\right)  $ then%
\[
M_{s_{i}\alpha,F}\left(  x^{\left(  1\right)  };\theta\right)  =-\frac
{t-z}{1-z}M_{\alpha,F}\left(  x^{\left(  1\right)  };\theta\right)  .
\]

\end{proposition}

\begin{proof}
From (\ref{step}) and (\ref{Tbp}) with $b\left(  x^{\left(  1\right)
},i\right)  =-1$ it follows that%
\begin{align*}
M_{s_{i}\alpha,F}\left(  x^{\left(  1\right)  };\theta\right)   &  =\left(
-1+\frac{1-t}{1-z}\right)  M_{\alpha,F}\left(  x^{\left(  1\right)  }%
;\theta\right)  +\left(  T_{i}+1\right)  M_{\alpha,F}\left(  x^{\left(
1\right)  }s_{i};\theta\right) \\
&  =-\frac{t-z}{1-z}M_{\alpha,F}\left(  x^{\left(  1\right)  };\theta\right)
,
\end{align*}
because $x^{\left(  1\right)  }s_{i}$ satisfies the hypotheses of the Lemma
implying $\left(  T_{i}+1\right)  M_{\alpha,F}\left(  x^{\left(  1\right)
}s_{i};\theta\right)  =0$.
\end{proof}

\begin{proposition}
\label{V1alpha}Suppose $\alpha\in\mathcal{N}_{1}$ then $M_{\alpha,F}\left(
x^{\left(  1\right)  };\theta\right)  =V^{\left(  1\right)  }\left(
\alpha\right)  \tau_{F}$ and%
\[
V^{\left(  1\right)  }\left(  \alpha\right)  =\left(  -1\right)
^{\mathrm{inv}\left(  \alpha\right)  }\mathcal{R}_{0}\left(  \alpha,F\right)
^{-1}V^{\left(  1\right)  }\left(  \alpha^{+}\right)  .
\]

\end{proposition}

\begin{proof}
By Lemma \ref{M1ctauF} $M_{\alpha,F}\left(  x^{\left(  1\right)  }%
;\theta\right)  $ is a multiple of $\tau_{F}$. For the product formula argue
by induction on $\mathrm{inv}\left(  \alpha\right)  $. If $\lambda
\in\mathbb{N}_{0}^{N,+}$ then $\mathcal{R}_{0}\left(  \lambda,F\right)  =1$.
If $\alpha_{i}<\alpha_{i+1}$ then%
\[
\frac{V^{\left(  1\right)  }\left(  s_{i}\alpha\right)  }{V^{\left(  1\right)
}\left(  \alpha\right)  }=-u_{0}\left(  \frac{\zeta_{\alpha,F}\left(
i+1\right)  }{\zeta_{\alpha,F}\left(  i\right)  }\right)  =-\frac
{\mathcal{R}_{0}\left(  \alpha,F\right)  }{\mathcal{R}_{0}\left(  s_{i}%
\alpha,F\right)  }.
\]

\end{proof}

We will use induction on the last nonzero part of $\lambda\in\mathbb{N}%
_{0}^{N,+}$ to derive $V^{\left(  \ast\right)  }\left(  \lambda\right)  $.
Suppose $\lambda_{k}\geq1$ and $\lambda_{i}=0$ for $i>k$ where $1\leq k\leq
N-m-1$ in type (0) and $1\leq k\leq m$ in type (1). Define compositions in
$\mathbb{N}_{0}^{N}$ by%
\begin{align}
\lambda^{\prime}  &  =\left(  \lambda_{1},\ldots,\lambda_{k-1},\lambda
_{k}-1,0,\ldots\right) \label{lb2a2b}\\
\alpha &  =\left(  \lambda_{k}-1,\lambda_{1},\ldots,\lambda_{k-1}%
,0,\ldots\right) \nonumber\\
\beta &  =\left(  \lambda_{1},\ldots,\lambda_{k-1},0,\ldots,\lambda_{k}\right)
\nonumber
\end{align}%
\begin{align}
\delta &  =\left(  \lambda_{1},\ldots,\lambda_{k-1},0,\ldots,\overset
{n}{\lambda_{k}},\overset{n+1}{0},0\ldots\right) \label{del2lb}\\
\lambda &  =\left(  \lambda_{1},\ldots,\lambda_{k-1},\lambda_{k}%
,0,\ldots\right)  ,\nonumber
\end{align}
where $n=N-m-1$ in type (0) and $n=m$ in type (1). The transitions from
$\lambda^{\prime}\rightarrow\alpha$ and from $\delta\rightarrow\lambda$ use
Propositions \ref{V0alpha} and \ref{V1alpha}. The affine step $\alpha
\rightarrow\beta$ and the steps $\beta\rightarrow\delta$ require technical computations.

\begin{proposition}
Suppose $\lambda\in\mathcal{N}_{0}^{+}$ and $\lambda^{\prime},\alpha$ are
given by (\ref{lb2a2b}) then%
\[
V^{\left(  0\right)  }\left(  \alpha\right)  =V^{\left(  0\right)  }\left(
\lambda^{\prime}\right)  \prod\limits_{i=1}^{k-1}\frac{1-q^{\lambda
_{i}-\lambda_{k}+1}t^{k-i}}{1-q^{\lambda_{i}-\lambda_{k}+1}t^{k-i+1}}.
\]

\end{proposition}

\begin{proof}
The spectral vector of $\lambda^{\prime}$ has $\zeta_{\lambda^{\prime}%
,E}\left(  i\right)  =q^{\lambda_{i}}t^{N-m-i}$ for $1\leq i<k,\zeta
_{\lambda^{\prime},E}\left(  k\right)  =q^{\lambda_{k}-1}t^{N-m-k}$ while
$\zeta_{\alpha,E}\left(  1\right)  =\zeta_{\lambda^{\prime},E}\left(
k\right)  $ and $\zeta_{\alpha,E}\left(  i\right)  =\zeta_{\lambda^{\prime}%
,E}\left(  i-1\right)  $ for $2\leq i\leq k$. The product is $\mathcal{R}%
_{1}\left(  \alpha,E\right)  ^{-1}.$
\end{proof}

\begin{proposition}
Suppose $\lambda\in\mathcal{N}_{0}^{+}$ and $\delta$ is as in (\ref{del2lb})
then%
\[
V^{\left(  0\right)  }\left(  \lambda\right)  =\frac{1-q^{\lambda_{k}%
}t^{N-m-k}}{1-q^{\lambda_{k}}t}V^{\left(  0\right)  }\left(  \delta\right)  .
\]

\end{proposition}

\begin{proof}
The relevant part of $\zeta_{\delta,E}$ is $\zeta_{\delta,E}\left(  i\right)
=t^{N-m-i-1}$ for $k\leq i\leq N-m-2$ and $\zeta_{\delta,E}\left(
N-m-1\right)  =q^{\lambda_{k}}t^{N-m-k}$. Thus%
\[
\mathcal{R}_{1}\left(  \delta,E\right)  =\prod\limits_{i=k}^{N-m-2}%
u_{1}\left(  q^{\lambda_{k}}t^{i+1-k}\right)  =\prod\limits_{j=0}%
^{N-m-k-2}\frac{1-q^{\lambda_{k}}t^{j+2}}{1-q^{\lambda_{k}}t^{j+1}}%
\]
and this product telescopes.
\end{proof}

\begin{proposition}
Suppose $\lambda\in\mathcal{N}_{1}^{+}$ and $\lambda^{\prime},\alpha$ are
given by (\ref{lb2a2b}) then%
\[
V^{\left(  1\right)  }\left(  \alpha\right)  =\left(  -t\right)  ^{1-k}%
\prod\limits_{i=1}^{k-1}\frac{1-q^{\lambda_{i}-\lambda_{k}+1}t^{i-k}%
}{1-q^{\lambda_{i}-\lambda_{k}+1}t^{i-k-1}}V^{\left(  1\right)  }\left(
\lambda^{\prime}\right)  .
\]

\end{proposition}

\begin{proof}
The spectral vector of $\lambda^{\prime}$ has $\zeta_{\lambda^{\prime}%
,F}\left(  i\right)  =q^{\lambda_{i}}t^{i-1-m}$ for $1\leq i<m,\zeta
_{\lambda^{\prime},F}\left(  k\right)  =q^{\lambda_{k}-1}t^{k-1-m}$ while
$\zeta_{\alpha,F}\left(  1\right)  =\zeta_{\lambda^{\prime},F}\left(
k\right)  $ and $\zeta_{\alpha,F}\left(  i\right)  =\zeta_{\lambda^{\prime}%
,F}\left(  i-1\right)  $ for $2\leq i\leq k$. Also $\mathrm{inv}\left(
\alpha\right)  =k-1$. Then%
\[
\mathcal{R}_{0}\left(  \alpha,F\right)  =\prod\limits_{i=1}^{k-1}u_{0}\left(
q^{\lambda_{i}-\lambda_{k}+1}t^{i-k}\right)  =t^{k-1}\prod\limits_{i=1}%
^{k-1}\frac{1-q^{\lambda_{i}-\lambda_{k}+1}t^{i-k-1}}{1-q^{\lambda_{i}%
-\lambda_{k}+1}t^{i-k}}.
\]
Combine this with $V^{\left(  1\right)  }\left(  \alpha\right)  =\left(
-1\right)  ^{\mathrm{inv}\left(  \alpha\right)  }\mathcal{R}_{0}\left(
\alpha,F\right)  ^{-1}V^{\left(  1\right)  }\left(  \lambda^{\prime}\right)
$.$.$
\end{proof}

\begin{proposition}
Suppose $\lambda\in\mathcal{N}_{1}^{+}$ and $\delta$ is as in (\ref{del2lb})
then%
\[
V^{\left(  1\right)  }\left(  \lambda\right)  =\left(  -t\right)  ^{m-k}%
\frac{1-q^{\lambda_{k}}t^{k-m-1}}{1-q^{\lambda_{k}}t^{-1}}V^{\left(  1\right)
}\left(  \delta\right)  .
\]

\end{proposition}

\begin{proof}
The relevant part of $\zeta_{\delta,F}$ is $\zeta_{\delta,F}\left(  i\right)
=t^{i-m}$ for $k\leq i\leq m-1$ and $\zeta_{\delta,E}\left(  m\right)
=q^{\lambda_{k}}t^{k-m-1}$. Thus%
\[
\mathcal{R}_{0}\left(  \delta,E\right)  =\prod\limits_{i=k}^{m-1}u_{0}\left(
q^{\lambda_{k}}t^{k-i-1}\right)  =t^{m-k}\prod\limits_{j=0}^{m-k-1}%
\frac{1-q^{\lambda_{k}}t^{-j-2}}{1-q^{\lambda_{k}}t^{-j-1}}%
\]
and this product telescopes to $\dfrac{1-q^{\lambda_{k}}t^{k-m-1}%
}{1-q^{\lambda_{k}}t^{-1}}$. The use of $\mathrm{inv}\left(  \delta\right)
=m-k$ completes the proof.
\end{proof}

The methods used in these calculations are similar to those used in
\cite{DL2015} for evaluations of scalar valued Macdonald polynomials, however
the following computations (from $\alpha$ to $\delta$) are significantly different.

Each of the remaining transitions is calculated in its own subsection. The
following two lemmas will be used in both types. Recall $t\boldsymbol{T}%
_{i}^{-1}=1-t+\boldsymbol{T}_{i}$ for any $i$.

\begin{lemma}
Suppose $f=t\boldsymbol{T}_{i}^{-1}g$ and $b=b\left(  x;i\right)  $ then
\begin{equation}
\left(  T_{i}-b\right)  f\left(  xs_{i}\right)  =\left(  1+b\right)  \left(
t-b\right)  g\left(  x\right)  -b\left(  T_{i}-b\right)  g\left(
xs_{i}\right)  \label{TbfTb}%
\end{equation}

\end{lemma}

\begin{proof}
From $g=t^{-1}\boldsymbol{T}_{i}f$ we get%
\begin{align*}
tg\left(  x\right)   &  =bf\left(  x\right)  +\left(  T_{i}-b\right)  f\left(
xs_{i}\right) \\
f\left(  x\right)   &  =\left(  1-t+b\right)  g\left(  x\right)  +\left(
T_{i}-b\right)  g\left(  xs_{i}\right)  ,
\end{align*}
thus%
\begin{align*}
\left(  T_{i}-b\right)  f\left(  xs_{i}\right)   &  =tg\left(  x\right)
-b\left(  1-t+b\right)  g\left(  x\right)  -b\left(  T_{i}-b\right)  g\left(
xs_{i}\right) \\
&  =\left(  1+b\right)  \left(  t-b\right)  g\left(  x\right)  -b\left(
T_{i}-b\right)  g\left(  xs_{i}\right)  .
\end{align*}

\end{proof}

The next formula is a modified braid relation.

\begin{lemma}
\label{Tbraids}Suppose $b=\dfrac{ac}{a+c+1-t}$ or $j,\ell,j+\ell\in
\mathbb{Z}\backslash\left\{  0\right\}  $ then%
\begin{align*}
\left(  T_{i}-a\right)  \left(  T_{i+1}-b\right)  \left(  T_{i}-c\right)   &
=\left(  T_{i+1}-c\right)  \left(  T_{i}-b\right)  \left(  T_{i+1}-a\right) \\
\left(  T_{i}-\kappa_{j}\right)  \left(  T_{i+1}-\kappa_{j+\ell}\right)
\left(  T_{i}-\kappa_{\ell}\right)   &  =\left(  T_{i+1}-\kappa_{\ell}\right)
\left(  T_{i}-\kappa_{j+\ell}\right)  \left(  T_{i+1}-\kappa_{j}\right)
\end{align*}

\end{lemma}

\begin{proof}
Expand%
\begin{align*}
&  \left(  T_{i}-a\right)  \left(  T_{i+1}-b\right)  \left(  T_{i}-c\right)
+cT_{i}T_{i+1}+aT_{i+1}T_{i}+abc\\
&  =T_{i}T_{i+1}T_{i}+acT_{i+1}+b\left(  a+c\right)  T_{i}-bT_{i}^{2}\\
&  =T_{i}T_{i+1}T_{i}-bt+acT_{i+1}+b\left(  a+c-t+1\right)  T_{i}\\
&  =T_{i}T_{i+1}T_{i}-bt+ac\left(  T_{i+1}+T_{i}\right)  .
\end{align*}
which is symmetric in $T_{i},T_{i+1}$ since $T_{i}T_{i+1}T_{i}=T_{i+1}%
T_{i}T_{i+1}$. If $a=\kappa_{j}$ and $c=\kappa_{\ell}$ then $b=\kappa_{j+\ell
}$.
\end{proof}

\subsection{From $\alpha$ to $\delta$ for type (0)}

In this section we will prove $M_{\delta,E}\left(  x^{\left(  0\right)
}\right)  =q^{\lambda_{k}-1}t^{2N-m-k-1}\dfrac{1-q^{\lambda_{k}}t^{N-k+1}%
}{1-q^{\lambda_{k}}t^{N-k}}M_{\alpha,E}\left(  x^{\left(  0\right)  }\right)
$. Start with $\delta$ (where $\delta_{i}=\lambda_{i}$ for $1\leq i\leq
k-1,\delta_{N-m-1}=\lambda_{k}$ and $\delta_{i}=0$ otherwise). Let
$\beta^{\left(  N-m-1\right)  }=\delta$ and $\beta^{\left(  j\right)
}=s_{j-1}\beta^{\left(  j-1\right)  }$ for $N-m\leq j\leq N$ (so that
$\beta^{\left(  N\right)  }=\beta$ in (\ref{lb2a2b})). Abbreviate
$z=\zeta_{\delta,E}\left(  N-m-1\right)  =\zeta_{\lambda,E}\left(  k\right)
=q^{\lambda_{k}}t^{N-m-k}$. If $N-m-1\leq i<N$ then $\zeta_{\beta^{\left(
i+1\right)  },E}\left(  i+1\right)  =z,\zeta_{\beta^{\left(  i+1\right)  }%
,E}\left(  i\right)  =t^{i+1-N}$. Set $p_{i}\left(  x\right)  =M_{\beta
^{\left(  i\right)  },E}\left(  x;\theta\right)  $ for $N-m-1\leq i\leq N$,
then%
\begin{align*}
p_{i}\left(  x\right)   &  =\left(  \boldsymbol{T}_{i-1}+\frac{1-t}%
{1-zt^{N-i-1}}\right)  p_{i+1}\left(  x\right) \\
&  =\left(  b\left(  x;i\right)  +\frac{1-t}{1-zt^{N-i-1}}\right)
p_{i+1}\left(  x\right)  +\left(  T_{i}-b\left(  x;i\right)  \right)
p_{i+1}\left(  xs_{i}\right)  .
\end{align*}
To start set $i=N-m-1$ and $x=x^{\left(  0\right)  }$ (thus $b\left(
x^{\left(  0\right)  };i\right)  =t$)%
\begin{equation}
p_{N-m-1}\left(  x^{\left(  0\right)  }\right)  =\frac{1-zt^{m+1}}{1-zt^{m}%
}p_{N-m}\left(  x^{\left(  0\right)  }\right)  +\left(  T_{N-m-1}-t\right)
p_{N-m}\left(  x^{\left(  0\right)  }s_{N-m-1}\right)  . \label{start00}%
\end{equation}
Two series of points are used in the calculation: Define $y_{N-m-1}=x^{\left(
0\right)  },y_{N-m}=x^{\left(  0\right)  }s_{N-m},y_{i}=y_{i-1}s_{i}$ for
$N-m+1\leq i\leq N-1$; define $v_{N-m-2}=x^{\left(  0\right)  },v_{i}%
=v_{i-1}s_{i}$ for $N-m-1\leq i\leq N-1$. Thus%
\begin{align*}
y_{i-1}  &  =\left(  \ldots,\overset{N-m-1}{t^{N-m-2}},t^{N-m},\ldots
,\overset{i}{t^{N-m-1}},t^{i},\ldots,t^{N-1}\right)  ,b\left(  y_{i-1}%
;i\right)  =\kappa_{i+m-N+1},\\
v_{i-1}  &  =\left(  \ldots,\overset{N-m-1}{t^{N-m-1}},t^{N-m},\ldots
,\overset{i}{t^{N-m-2}},t^{i},\ldots,t^{N-1}\right)  ,b\left(  v_{i-1}%
;i\right)  =\kappa_{i+m-N+2}.
\end{align*}

\begin{lemma}
\label{Tkpy0}Suppose $N-m\leq j<N$ and $\boldsymbol{T}_{j}p=-p$ then $\left(
T_{j}-\kappa_{j-N+m+1}\right)  p\left(  y_{j}\right)  =-\dfrac{\left[
j-N+m+2\right]  _{t}}{\left[  j-N+m+1\right]  _{t}}p\left(  y_{j-1}\right)  $
and $\left(  T_{j}-\kappa_{j-N+m+2}\right)  p\left(  v_{j}\right)
=-\dfrac{\left[  j-N+m+3\right]  _{t}}{\left[  j-N+m+2\right]  _{t}}p\left(
v_{j-1}\right)  $ for $1\leq j\leq i$.
\end{lemma}

\begin{proof}
This follows from
\[
\left(  T_{j}-b\left(  x;j\right)  \right)  p\left(  xs_{j}\right)  =-\left(
1+b\left(  x;j\right)  \right)  p\left(  x\right)
\]
with $x=y_{j-1}$ so that $xs_{j}=y_{j}=y_{j-1}$, and with $x=v_{j-1}$. If
$n=1,2,\ldots$ then $1+\kappa_{n}=\dfrac{1+t+\cdots+t^{n-1}+t^{n}}%
{1+t+\cdots+t^{n-1}}\allowbreak=\dfrac{\left[  n+1\right]  _{t}}{\left[
n\right]  _{t}}$.
\end{proof}

\begin{proposition}
\label{P0pTp}For $N-m\leq i\leq N$%
\begin{align}
p_{N-m-1}\left(  x^{0}\right)   &  =\frac{1-zt^{m+1}}{1-zt^{m}}\left(
T_{N-m}-t\right)  \cdots\left(  T_{i-1}-\kappa_{i-N+m}\right)  p_{i}\left(
y_{i-1}\right) \label{px0pi}\\
&  +\left(  T_{N-m-1}-t\right)  \left(  T_{N-m}-\kappa_{2}\right)
\cdots\left(  T_{i-1}-\kappa_{i-N+m+1}\right)  p_{i}\left(  v_{i-1}\right)
.\nonumber
\end{align}

\end{proposition}

\begin{proof}
The formula is true for $i=N-m$ by (\ref{start00}). Assume it holds for some
$i$ then%
\begin{align}
p_{i}\left(  y_{i-1}\right)   &  =\left(  \frac{1-t}{1-zt^{N-i-1}}%
+\kappa_{i+m-N+1}\right)  p_{i+1}\left(  y_{i-1}\right)  +\left(  T_{i}%
-\kappa_{i+m-N+1}\right)  p_{i+1}\left(  y_{i}\right) \label{step001}\\
p_{i}\left(  v_{i-1}\right)   &  =\left(  \frac{1-t}{1-zt^{N-i-1}}%
+\kappa_{i+m-N+2}\right)  p_{i+1}\left(  v_{i-1}\right)  +\left(  T_{i}%
-\kappa_{i+m-N+2}\right)  p_{i+1}\left(  v_{i}\right)  , \label{step002}%
\end{align}
and
\begin{align*}
\frac{1-t}{1-zt^{N-i-1}}+\kappa_{i+m-N+1}  &  =\frac{1-zt^{m}}{\left(
1-zt^{N-i-1}\right)  \left[  i+m-N+1\right]  _{t}},\\
\frac{1-t}{1-zt^{N-i-1}}+\kappa_{i+m-N+2}  &  =\frac{1-zt^{m+1}}{\left(
1-zt^{N-i-1}\right)  \left[  i+m-N+2\right]  _{t}}.
\end{align*}
Then $p_{i+1}$ satisfies the hypotheses of Lemma \ref{Tkpy0} for $N-m\leq j<i$
and
\begin{align}
&  \left(  T_{N-m}-t\right)  \cdots\left(  T_{i-1}-\kappa_{i-N+m}\right)
p_{i+1}\left(  y_{i-1}\right) \label{repTp0a}\\
&  =\left(  -1\right)  ^{i-N+m}\frac{\left[  2\right]  _{t}}{\left[  1\right]
_{t}}\frac{\left[  3\right]  _{t}}{\left[  2\right]  _{t}}\cdots\dfrac{\left[
i-N+m+1\right]  _{t}}{\left[  i-N+m\right]  _{t}}p_{i+1}\left(  y_{N-m-1}%
\right) \\
&  =\left(  -1\right)  ^{i-N+m}\left[  i-N+m+1\right]  _{t}~p_{i+1}\left(
x^{\left(  0\right)  }\right) \nonumber
\end{align}%
\begin{align}
&  \left(  T_{N-m-1}-t\right)  \cdots\left(  T_{i-1}-\kappa_{i-N+m+1}\right)
p_{i}\left(  v_{i-1}\right) \label{repTp0b}\\
&  =\left(  -1\right)  ^{i-N+m+1}\frac{\left[  2\right]  _{t}}{\left[
1\right]  _{t}}\frac{\left[  3\right]  _{t}}{\left[  2\right]  _{t}}%
\cdots\dfrac{\left[  i-N+m+2\right]  _{t}}{\left[  i-N+m+1\right]  _{t}%
}p_{i+1}\left(  v_{N-m-2}\right) \\
&  =\left(  -1\right)  ^{i-N+m+1}\left[  i-N+m+2\right]  _{t}~p_{i+1}\left(
x^{\left(  0\right)  }\right)  .\nonumber
\end{align}
The first part of the right side of (\ref{step001}) combined with
\ref{repTp0a} gives%
\begin{align*}
&  \frac{1-zt^{m+1}}{1-zt^{m}}\frac{1-zt^{m}}{\left(  1-zt^{N-i-1}\right)
\left[  i+m-N+1\right]  _{t}}\left(  -1\right)  ^{i-N+m}\left[
i-N+m+1\right]  _{t}~p_{i+1}\left(  x^{\left(  0\right)  }\right) \\
&  =\left(  -1\right)  ^{i-N+m}\frac{1-zt^{m+1}}{1-zt^{N-i-1}}p_{i+1}\left(
x^{\left(  0\right)  }\right)
\end{align*}
which cancels out the first part of the right side of (\ref{step002}) combined
with (\ref{repTp0b}); note the factor $\left(  -1\right)  ^{i-N+m+1}$. The
terms that remain are exactly the claimed formula for $i+1$.
\end{proof}

\begin{proposition}
\label{MzM0}$M_{\delta,E}\left(  x^{\left(  0\right)  };\theta\right)
=\dfrac{1-zt^{m+1}}{1-zt^{m}}\left(  T_{N-m}-t\right)  \cdots\left(
T_{N-1}-\kappa_{m}\right)  M_{\beta,E}\left(  y_{N-1};\theta\right)  $.
\end{proposition}

\begin{proof}
Set $i=N$ in (\ref{px0pi}). To complete the proof we need to show%
\[
\left(  T_{N-m-1}-t\right)  \left(  T_{N-m}-\kappa_{2}\right)  \cdots\left(
T_{N-1}-\kappa_{m+1}\right)  p_{N}\left(  v_{N-1}\right)  =0.
\]
By construction $\left(  v_{N-1}\right)  _{i}=t^{i}$ and $r_{\beta}\left(
i\right)  =i+1,$ $\zeta_{\beta,E}\left(  i\right)  =t^{i+1-N}$ for $N-m-1\leq
i\leq N-1$ . This implies $\left(  \boldsymbol{T}_{i}+1\right)  p_{N}=0$ and
$\left(  T_{i}+1\right)  p_{N}\left(  v_{N-1}\right)  =0$ for $N-m-1\leq
i<N-1$. Let $\psi_{1}=D\left(  \theta_{N-m-1}\theta_{N-m}\cdots\theta
_{N-1}\right)  $ then $\left(  T_{i}+1\right)  \psi_{1}=0$ for $N-m-1\leq
i<N-1.$ This property defines $\psi_{1}$ up to a multiplicative constant, and
thus $p_{N}\left(  v_{N-1}\right)  =c\psi_{1}$ (because $\psi_{2}%
:=T_{N-m-1}T_{N-m}\cdots T_{N-1}\psi_{1}$ satisfies $\left(  T_{j}+1\right)
\psi_{2}=0$ for $N-m\leq j<N$ and thus $\psi_{2}=c^{\prime}\tau_{E}$). To set
up an inductive argument let $\pi_{i}:=\theta_{N-m-1}\cdots\theta_{i-1}%
\theta_{i+1}\cdots\theta_{N}$ and set $f_{N}:=\pi_{N-1},f_{N-j}:=\left(
T_{N-j}-\kappa_{m+2-j}\right)  f_{N-j+1}$ for $1\leq j\leq m+1$. Then
$T_{j}\pi_{i}=-\pi_{i}$ if $N-m-1\leq j\leq i-2$ or $i+1\leq j<N$, $T_{i-1}%
\pi_{i}=\pi_{i-1}$ and $T_{i}\pi_{i}=\left(  t-1\right)  \pi_{i}+t\pi_{i+1}$.
Claim that%
\[
f_{N-j}=\pi_{N-j}+\frac{1}{\left[  m+2-j\right]  _{t}}\sum_{i=0}^{j-1}\left(
-1\right)  ^{i-j}t^{m+1-i}\pi_{N-i}.
\]
The first step is $f_{N-1}=\left(  T_{N-1}-\kappa_{m+1}\right)  \pi_{N}%
=\pi_{N-1}-\dfrac{t^{m+1}}{\left[  m+1\right]  _{t}}\pi_{N}$. Note $\kappa
_{n}+1=\frac{\left[  n+1\right]  _{t}}{\left[  n\right]  _{t}}$ for $n\geq1$.
Suppose the formula is true for some $j<m+1$, then%
\begin{align*}
f_{N-j-1}  &  =\left(  T_{N-j-1}-\kappa_{m+1-j}\right)  f_{N-j}=\pi
_{N-j-1}-\frac{t^{m+1-j}}{\left[  m+1-j\right]  _{t}}\pi_{N-j}\\
&  +\frac{1}{\left[  m+2-j\right]  _{t}}\sum_{i}^{j-1}\left(  -1\right)
^{i-j}t^{m+1-i}\left(  -\kappa_{m+1-j}-1\right)  \pi_{N-i}\\
&  =\pi_{N-j-1}+\frac{1}{\left[  m+1-j\right]  _{t}}\sum_{i}^{j}\left(
-1\right)  ^{i-j-1}t^{m+1-i}\pi_{N-i}.
\end{align*}
This proves the formula. Set $j=m+1$ then $f_{N-m-1}=\sum_{i=0}^{m+1}\left(
-1\right)  ^{m+1-i}t^{m+1-i}\pi_{N-i}$. By definition%
\[
D\left(  \pi_{N}\theta_{N}\right)  =\sum_{j=N-m-1}^{N}\left(  -1\right)
^{j-N+m+1}t^{j-1}\pi_{j}=\sum_{i=0}^{m+1}\left(  -1\right)  ^{m+1-i}%
t^{N-1-i}\pi_{N-i}%
\]
and so $f_{N-m-1}=t^{m-N+2}D\left(  \pi_{N}\theta_{N}\right)  $. Now
$p_{N}\left(  v_{N-1}\right)  =c\psi_{1}=cD\left(  \pi_{N}\right)  $Thus
\begin{align*}
&  \left(  T_{N-m-1}-t\right)  \left(  T_{N-m}-\kappa_{2}\right)
\cdots\left(  T_{N-1}-\kappa_{m+1}\right)  D\left(  \pi_{N}\right) \\
&  =D\left\{  \left(  T_{N-m-1}-t\right)  \left(  T_{N-m}-\kappa_{2}\right)
\cdots\left(  T_{N-1}-\kappa_{m+1}\right)  \pi_{N}\right\}  =D\left(
f_{N-m-1}\right)  =0,
\end{align*}
because $D^{2}=0$.
\end{proof}

Next we consider the transition from $\alpha$ to $\beta$ (see \ref{lb2a2b})
with the affine step $M_{\beta,E}\left(  x;\theta\right)  =x_{N}%
\boldsymbol{w}M_{\alpha,E}\left(  x;\theta\right)  $ (recall $\boldsymbol{w}%
p\left(  x;\theta\right)  =T_{N-1}\cdots T_{1}p\left(  qx_{N},x_{1}%
,\ldots,x_{N-1};\theta\right)  $). To get around the problem of evaluation at
the $q$-shifted point we use $\xi_{1}=t^{1-N}\boldsymbol{T}_{1}\boldsymbol{T}%
_{2}\cdots\boldsymbol{T}_{N-1}w$ thus%
\[
M_{\beta,E}\left(  x;\theta\right)  =x_{N}t^{N-1}\left(  \boldsymbol{T}%
_{N-1}^{-1}\cdots\boldsymbol{T}_{2}^{-1}\boldsymbol{T}_{1}^{-1}\xi
_{1}M_{\alpha,E}\right)  \left(  x;\theta\right)
\]
where $\xi_{1}M_{\alpha,E}=\zeta_{\alpha,E}\left(  1\right)  M_{\alpha
,E}=q^{\lambda_{k}-1}t^{N-m-k}M_{\alpha,E}$. From the previous formula we see
that we need to evaluate the right hand side at $x=y_{N-1}$ and apply $\left(
T_{N-m}-t\right)  \cdots\left(  T_{N-1}-\kappa_{m}\right)  $. Since
$\zeta_{\alpha,E}\left(  i\right)  =t^{i-N}$ for $N-m\leq i\leq N$ it follows
that $\left(  \boldsymbol{T}_{i}+1\right)  M_{\alpha,E}=0$ for $N-m\leq i<N$.

\begin{definition}
Let $r_{0}=M_{\alpha,E}$ and $r_{i}=t\boldsymbol{T}_{i}^{-1}r_{i-1}$ for
$1\leq i<N$.
\end{definition}

The corresponding evaluation formula is%
\begin{equation}
r_{i}\left(  x\right)  =\left(  1-t+b\left(  x;i\right)  \right)
r_{i-1}\left(  x\right)  +\left(  T_{i}-b\left(  x;i\right)  \right)
r_{i-1}\left(  xs_{i}\right)  . \label{r2r}%
\end{equation}

\begin{proposition}
Suppose $1\leq i\leq N-m-2$ then $r_{i}\left(  x^{\left(  0\right)  }\right)
=r_{0}\left(  x^{\left(  0\right)  }\right)  =M_{\alpha,E}\left(  x^{\left(
0\right)  }\right)  .$
\end{proposition}

\begin{proof}
From $\left(  \boldsymbol{T}_{i}+1\right)  M_{\alpha,E}=0$ for $N-m\leq i<N$
it follows that $\left(  \boldsymbol{T}_{i}+1\right)  r_{j}=0$ if $j<N-m-1$.
By (\ref{r2r}) $r_{\ell}\left(  x^{\left(  0\right)  }\right)  =r_{\ell
-1}\left(  x^{\left(  0\right)  }\right)  +\left(  T_{\ell}-t\right)
r_{\ell-1}\left(  x^{\left(  0\right)  }s_{\ell}\right)  $ . Suppose
$\ell<N-m-1$ then $x^{\left(  0\right)  }s_{\ell}$ satisfies $\left(
x^{\left(  0\right)  }s_{\ell}\right)  _{i}=t^{i-1}$ for $i\geq N-m$ so that
$b\left(  x^{\left(  0\right)  }s_{\ell};i\right)  =t$ and $\left(
T_{i}+1\right)  r_{\ell-1}\left(  x^{\left(  0\right)  }s_{\ell}\right)  =0$,
$r_{\ell-1}\left(  x^{\left(  0\right)  }s_{\ell}\right)  $ is a multiple of
$\tau_{E}$ and $\left(  T_{\ell}-t\right)  r_{\ell-1}\left(  x^{\left(
0\right)  }s_{\ell}\right)  =0$. Thus $r_{\ell}\left(  x^{\left(  0\right)
}\right)  =r_{\ell-1}\left(  x^{\left(  0\right)  }\right)  $ and this holds
for $1\leq\ell\leq N-m-2$.
\end{proof}

Recall the points $y_{i}$ given by $y_{N-m-1}=x^{\left(  0\right)  }%
,y_{i}=y_{i-1}s_{i}$ for $N-m\leq i<N$. Define $\widetilde{y}_{i}=y_{i}%
s_{N-1}s_{N-2}\cdots s_{i+1}.$ for $N-m-1\leq i<N$. By the braid relations
\begin{align}
\widetilde{y}_{i+1}s_{i+1}s_{i+2}  &  =\left(  y_{i+1}s_{N-1}\cdots
s_{i+2}\right)  s_{i+1}s_{i+2}=y_{i+1}s_{N-1}\cdots s_{i+3}s_{i+1}%
s_{i+2}s_{i+1}\label{yi2y}\\
&  =y_{i+1}s_{i+1}s_{N-1}\cdots s_{i+2}s_{i+1}=y_{i}s_{N-1}\cdots
s_{i+1}=\widetilde{y}_{i}.\nonumber
\end{align}
These products are used in the proofs:%
\begin{align*}
P_{N-j}  &  =\left(  T_{N-m}-t\right)  \left(  T_{N-m+1}-\kappa_{2}\right)
\cdots\left(  T_{N-j}-\kappa_{m+1-j}\right) \\
\widetilde{P}_{N-j}  &  =\left(  T_{N-1}-t\right)  \left(  T_{N-2}-\kappa
_{2}\right)  \cdots\left(  T_{N-j}-\kappa_{j}\right)  .
\end{align*}
If $i+1<j$ then $P_{N-j}$ commutes with $\widetilde{P}_{N-i}$.

\begin{lemma}
\label{Pnjf}Suppose $\boldsymbol{T}f=-f$ for $N-j\leq i<N$ and $u_{i}=t^{i-1}$
for $N-j+1\leq i\leq N$ then%
\[
\widetilde{P}_{N-j}f\left(  us_{N-1}s_{N-2}\cdots s_{N-j}\right)  =\left(
-1\right)  ^{j}\left[  j+1\right]  _{t}f\left(  u\right)  .
\]

\end{lemma}

\begin{proof}
Let $\widetilde{u}^{\left(  n\right)  }=us_{N-1}\cdots s_{N-n}$ then $\left(
\widetilde{u}^{\left(  n-1\right)  }\right)  _{N-n+1}=t^{N-1},\left(
\widetilde{u}^{\left(  n-1\right)  }\right)  _{N-n}=t^{N-n}$ and $b\left(
\widetilde{u}^{\left(  n-1\right)  };N-n\right)  =\kappa_{n-1}$. Thus%
\begin{align*}
\left(  T_{N-n}-\kappa_{n-1}\right)  f\left(  \widetilde{u}^{\left(  n\right)
}\right)   &  =\left(  T_{N-n}-\kappa_{n-1}\right)  f\left(  \widetilde
{u}^{\left(  n-1\right)  }s_{N-n}\right) \\
&  =-\left(  1+\kappa_{n-1}\right)  f\left(  \widetilde{u}^{\left(
n-1\right)  }\right)  =-\dfrac{\left[  n\right]  _{t}}{\left[  n-1\right]
_{t}}f\left(  \widetilde{u}^{\left(  n-1\right)  }\right)  .
\end{align*}
Repeated application of this relation shows $\widetilde{P}_{N-j}f\left(
us_{N-1}s_{N-2}\cdots s_{N-j}\right)  =\left(  -1\right)  ^{j}\frac{1}{\left[
2\right]  _{t}}\frac{\left[  2\right]  _{t}}{\left[  3\right]  _{t}}%
\cdots\frac{\left[  j+1\right]  _{t}}{\left[  j\right]  _{t}}f\left(
u\right)  $.
\end{proof}

If $x_{i+1}=t^{j}x_{ii}$ then $1+b\left(  x;i\right)  =\dfrac{\left[
j+1\right]  _{t}}{\left[  j\right]  _{t}},1-t+b\left(  x;i\right)  =\dfrac
{1}{\left[  j\right]  _{t}}$ and $\left(  1+b\left(  x;i\right)  \right)
\left(  t-b\left(  x;i\right)  \right)  =\dfrac{t\left[  j+1\right]
_{t}\left[  j-1\right]  _{t}}{\left[  j\right]  _{t}^{2}}$.

\begin{proposition}
For $2\leq j\leq m+1$%
\begin{align}
P_{N-1}r_{N-1}\left(  y_{N-1}\right)   &  =t^{j-1}\frac{\left[  m+1\right]
_{t}\left[  m+1-j\right]  _{t}}{\left[  m\right]  _{t}\left[  m+2-j\right]
_{t}}P_{N-j}r_{N-j}\left(  y_{N-j}\right) \label{PrPr0a}\\
&  +\left(  -1\right)  ^{j-1}\frac{t^{m}}{\left[  m+2-j\right]  _{t}%
}\widetilde{P}_{N-j+2}P_{N-j}\left(  T_{N-j+1}-\kappa_{m}\right)
r_{N-j}\left(  \widetilde{y}_{N-j}\right)  . \label{PrPr0b}%
\end{align}

\end{proposition}

\begin{proof}
Proceed by induction. By (\ref{TbfTb})%
\begin{gather*}
P_{N-1}r_{N-1}\left(  y_{N-1}\right)  =P_{N-2}\left(  T_{N-1}-\kappa
_{m}\right)  r_{N-1}\left(  y_{N-2}s_{N-1}\right) \\
=\dfrac{t\left[  m+1\right]  _{t}\left[  m-1\right]  _{t}}{\left[  m\right]
_{t}^{2}}P_{N-2}r_{N-2}\left(  y_{N-2}\right)  -\frac{t^{m}}{\left[  m\right]
_{t}}P_{N-2}\left(  T_{N-1}-\kappa_{m}\right)  r_{N-2}\left(  y_{N-2}%
s_{N-1}\right)
\end{gather*}
and $y_{N-1}=y_{N-2}s_{N-1}=\widetilde{y}_{N-1}$. Thus the formula is valid
for $j=2$ (with $\widetilde{P}_{N}=1$). Suppose it holds for some $j\leq m$,
then $b\left(  y_{N-1-j};N-j\right)  =\kappa_{m+1-j}$ and%
\begin{gather*}
P_{N-j}r_{N-j}\left(  y_{N-j}\right)  =P_{N-j-1}\left(  T_{N-j}-\kappa
_{m+1-j}\right)  r_{N-j}\left(  y_{N-1-j}s_{N-j}\right) \\
=\dfrac{t\left[  m+2-j\right]  _{t}\left[  m-j\right]  _{t}}{\left[
m+1-j\right]  _{t}^{2}}P_{N-j-1}r_{N-j-1}\left(  y_{N-1-j}\right) \\
-\frac{t^{m+1-j}}{\left[  m+1-j\right]  _{t}}P_{N-j-1}\left(  T_{N-j}%
-\kappa_{m+1-j}\right)  r_{N-j-1}\left(  y_{N-j}\right)
\end{gather*}
Combine with formula (\ref{PrPr0a}) to obtain%
\begin{equation}
t^{j}\frac{\left[  m+1\right]  _{t}\left[  m-j\right]  _{t}}{\left[  m\right]
_{t}\left[  m+1-j\right]  _{t}}P_{N-j-1}r_{N-j-1}\left(  y_{N-j-1}\right)
-\frac{t^{m}\left[  m+1\right]  _{t}}{\left[  m\right]  _{t}\left[
m+2-j\right]  _{t}}P_{N-j}r_{N-j-1}\left(  y_{N-j}\right)  . \label{AAA}%
\end{equation}
For the part in (\ref{PrPr0b}) $b\left(  \widetilde{y}_{N-j};N-j\right)
=\kappa_{j-1}$ thus%
\[
r_{N-j}\left(  \widetilde{y}_{N-j}\right)  =\frac{1}{\left[  j-1\right]  _{t}%
}r_{N-j-1}\left(  \widetilde{y}_{N-j}\right)  +\left(  T_{N-j}-\kappa
_{j-1}\right)  r_{N-j-1}\left(  \widetilde{y}_{N-j}s_{N-j}\right)  .
\]
The first part leads to%
\begin{align*}
&  \left(  -1\right)  ^{j-1}\frac{t^{m}}{\left[  m+2-j\right]  _{t}\left[
j-1\right]  _{t}}\widetilde{P}_{N-j+2}P_{N-j}\left(  T_{N-j+1}-\kappa
_{m}\right)  r_{N-j-1}\left(  \widetilde{y}_{N-j}\right) \\
&  =\left(  -1\right)  ^{j}\frac{\left[  m+1\right]  _{t}}{\left[  m\right]
_{t}}\frac{t^{m}}{\left[  m+2-j\right]  _{t}\left[  j-1\right]  _{t}%
}\widetilde{P}_{N-j+2}P_{N-j}r_{N-j-1}\left(  \widetilde{y}_{N-j}%
s_{N-j+1}\right)
\end{align*}
because $b\left(  \widetilde{y}_{N-j}s_{N-j+1};N-j+1\right)  =\kappa_{m}$ and
$\left(  \boldsymbol{T}_{N-j+1}+1\right)  r_{N-j-1}=0$. Then%
\begin{align*}
\widetilde{P}_{N-j+2}P_{N-j}r_{N-j-1}\left(  \widetilde{y}_{N-j}%
s_{N-j+1}\right)   &  =P_{N-j}\widetilde{P}_{N-j+2}r_{N-j-1}\left(
y_{N-j}s_{N-1}\cdots s_{N-j+2}\right) \\
&  =\left(  -1\right)  ^{j}\left[  j-1\right]  _{t}P_{N-j}r_{N-j-1}\left(
y_{N-j}\right)
\end{align*}
by Lemma \ref{Pnjf}, so combine to obtain $\dfrac{t^{m}\left[  m+1\right]
_{t}}{\left[  m+2-j\right]  _{t}\left[  m\right]  _{t}}P_{N-j}r_{N-j-1}\left(
y_{N-j}\right)  $ which cancels the second term in (\ref{AAA}). The second
part gives%
\begin{align*}
&  \left(  -1\right)  ^{j-1}\frac{t^{m}}{\left[  m+2-j\right]  _{t}}%
\widetilde{P}_{N-j+2}P_{N-j}\left(  T_{N-j+1}-\kappa_{m}\right)  \left(
T_{N-j}-\kappa_{j-1}\right)  r_{N-j-1}\left(  \widetilde{y}_{N-j}%
s_{N-j}\right) \\
&  =\frac{\left(  -1\right)  ^{j-1}t^{m}}{\left[  m+2-j\right]  _{t}%
}\widetilde{P}_{N-j+2}P_{N-j-1}\left(  T_{N-j}-\kappa_{m-j+1}\right) \\
&  \times\left(  T_{N-j+1}-\kappa_{m}\right)  \left(  T_{N-j}-\kappa
_{j-1}\right)  r_{N-j-1}\left(  \widetilde{y}_{N-j}s_{N-j}\right) \\
&  =\frac{\left(  -1\right)  ^{j-1}t^{m}}{\left[  m+2-j\right]  _{t}%
}\widetilde{P}_{N-j+2}P_{N-j-1}\left(  T_{N-j+1}-\kappa_{j-1}\right) \\
&  \times\left(  T_{N-j}-\kappa_{m}\right)  \left(  T_{N-j+1}-\kappa
_{m-j+1}\right)  r_{N-j-1}\left(  \widetilde{y}_{N-j}s_{N-j}\right) \\
&  =\frac{\left(  -1\right)  ^{j}t^{m}\left[  m+2-j\right]  _{t}}{\left[
m+1-j\right]  _{t}\left[  m+2-j\right]  _{t}}\widetilde{P}_{N-j+2}%
P_{N-j-1}\left(  T_{N-j+1}-\kappa_{j-1}\right) \\
&  \times\left(  T_{N-j}-\kappa_{m}\right)  r_{N-j-1}\left(  \widetilde
{y}_{N-j}s_{N-j}s_{-j+1}\right) \\
&  =\frac{\left(  -1\right)  ^{j}t^{m}}{\left[  m+1-j\right]  _{t}}%
\widetilde{P}_{N-j+1}P_{N-j-1}\left(  T_{N-j}-\kappa_{m}\right)
r_{N-j-1}\left(  \widetilde{y}_{N-1-j}\right)
\end{align*}
by Lemma \ref{Tbraids}, formula (\ref{yi2y}) and $b\left(  \widetilde{y}%
_{N-j}s_{N-j};N-j+1\right)  =\kappa_{m-j+1}$.
\end{proof}

\begin{proposition}
$P_{N-1}r_{N-1}\left(  y_{N-1}\right)  =t^{m}M_{\alpha,E}\left(  x^{\left(
0\right)  }\right)  $.
\end{proposition}

\begin{proof}
Set $j=m+1$ in (\ref{PrPr0a}) thus%
\begin{gather*}
P_{N-1}r_{N-1}\left(  y_{N-1}\right)  =\left(  -1\right)  ^{m}t^{m}%
\widetilde{P}_{N-m+1}\left(  T_{N-m}-\kappa_{m}\right)  r_{N-m-1}\left(
\widetilde{y}_{N-m-1}\right)  \\
=\left(  -1\right)  ^{m}t^{m}\widetilde{P}_{N-m}\left\{
\begin{array}
[c]{c}%
\frac{1}{\left[  m+1\right]  _{t}}r_{N-m-2}\left(  \widetilde{y}%
_{N-m-1}\right)  \\
+\left(  T_{N-m-1}-\kappa_{m+1}\right)  r_{N-m-2}\left(  \widetilde{y}%
_{N-m-1}s_{N-m-1}\right)
\end{array}
\right\}
\end{gather*}
By Lemma \ref{Pnjf} $\widetilde{P}_{N-m}p_{r-m-2}\left(  y_{N-m-1}%
s_{N-1}\cdots s_{N-m}\right)  =\left(  -1\right)  ^{m}\left[  m+1\right]
_{t}r_{N-m-2}\left(  y_{N-m-1}\right)  $. Furthermore%
\[
\widetilde{y}_{N-m-1}s_{N-m+1}=\left(  \ldots,\overset{N-m-1}{t^{N-1}%
},\overset{N-m}{t^{N-m-2}},\overset{N-m+1}{t^{N-m-1}},\ldots,t^{N-2}\right)  ,
\]
and thus $\left(  \boldsymbol{T}_{i}+1\right)  r_{N-m-2}=0$ and by Lemma
\ref{Ti+1} $\left(  T_{i}+1\right)  r_{N-m-2}\left(  \widetilde{y}%
_{N-m-1}s_{N-m-1}\right)  =0$ for $N-m\leq i<N$. This implies $r_{N-m-2}%
\left(  \widetilde{y}_{N-m-1}s_{N-m-1}\right)  =c\tau_{E}$. But $\widetilde
{P}_{N-m-1}\tau_{E}=0$ and this is proved by an argument like the one used in
Proposition \ref{MzM0}. Let $\pi_{i}:=\theta_{N-m-1}\cdots\theta_{i-1}%
\theta_{i+1}\cdots\theta_{N}$ and set $f_{0}:=\pi_{N-m-1},f_{j}:=\left(
T_{N-m-2+j}-\kappa_{m+2-j}\right)  f_{j-1}$ for $1\leq j\leq m+1$. Claim%
\[
f_{j}=t^{j}\pi_{N-m-1+j}+\frac{1}{\left[  m+2-j\right]  _{t}}\sum_{i=0}%
^{j-1}\left(  -1\right)  ^{j-i}t^{i}\pi_{N-m-1+i}.
\]
The first step is%
\begin{align*}
f_{1} &  =\left(  T_{N-m-1}-\kappa_{m+1}\right)  \pi_{N-m-1}=t\pi
_{N-m}+\left(  t-1-\kappa_{m+1}\right)  \pi_{N-m--1}\\
&  =t\pi_{N-m}-\frac{1}{\left[  m+1\right]  _{t}}\pi_{N-m-1}.
\end{align*}
Suppose the formula is true for some $j<m+1$ then%
\begin{align*}
f_{j+1} &  =\left(  T_{N-m-1+j}-\kappa_{m+1-j}\right)  f_{j}=t^{j+1}%
\pi_{N-m+j}+t^{j}\left(  t-1-\kappa_{m+1-j}\right)  \pi_{N-m-1+j}\\
&  +\frac{1}{\left[  m+2-j\right]  _{t}}\left(  -1-\kappa_{m+1-j}\right)
\sum_{i=0}^{j-1}\left(  -1\right)  ^{j-i}t^{i}\pi_{N-m-1+i}\\
&  =t^{j+1}\pi_{N-m+j}-\frac{1}{\left[  m+1-j\right]  _{t}}t^{j}\pi
_{N-m-1+j}+\frac{1}{\left[  m+1-j\right]  _{t}}\sum_{i=0}^{j-1}\left(
-1\right)  ^{j+1-i}t^{i}\pi_{N-m-1+i},
\end{align*}
and this is the formula for $j+1$. Then $f_{m+1}=\sum\limits_{i=0}%
^{m+1}\left(  -1\right)  ^{m+1-i}t^{i}\pi_{N-m-1+i}$ and $D\left(  \pi
_{N}\theta_{N}\right)  =\sum\limits_{i=N-m-1}^{N}t^{i-1}\left(  -1\right)
^{i-N+m+1}\pi_{i}$, and thus $f_{m+1}=\left(  -1\right)  ^{m+1}t^{N-m-2}%
D\left(  \pi_{N}\theta_{N}\right)  =0$ and $D\left(  f_{m+1}\right)  =0$. As
in Proposition \ref{MzM0} this implies $\widetilde{P}_{N-m-1}r_{N-m-2}\left(
\widetilde{y}_{N-m-1}s_{N-m-1}\right)  =0$, and this completes the proof.
\end{proof}

\subsection{Evaluation formula for type (0)}

Recall the intermediate steps:%
\begin{align*}
V^{\left(  0\right)  }\left(  \alpha\right)   &  =V^{\left(  0\right)
}\left(  \lambda^{\prime}\right)  \prod\limits_{i=1}^{k-1}\frac{1-q^{\lambda
_{i}-\lambda_{k}+1}t^{k-i}}{1-q^{\lambda_{i}-\lambda_{k}+1}t^{k-i+1}}\\
P_{N-1}r_{N-1}\left(  y_{N-1}\right)   &  =t^{m}M_{\alpha,E}\left(  x^{\left(
0\right)  };\theta\right) \\
M_{\beta,E}\left(  y_{N-1};\theta\right)   &  =\zeta_{\alpha,E}\left(
1\right)  \left(  y_{N-1}\right)  _{N}~r_{N-1}\left(  y_{N-1}\right) \\
M_{\delta,E}\left(  x^{\left(  0\right)  };\theta\right)   &  =\dfrac
{1-q^{\lambda_{k}}t^{N-k+1}}{1-q^{\lambda_{k}}t^{N-k}}P_{N-1}M_{\beta
,E}\left(  y_{N-1};\theta\right) \\
V^{\left(  0\right)  }\left(  \lambda\right)   &  =\frac{1-q^{\lambda_{k}%
}t^{N-m-k}}{1-q^{\lambda_{k}}t}V^{\left(  0\right)  }\left(  \delta\right)  .
\end{align*}

\begin{proposition}
Suppose $\lambda\in\mathcal{N}_{0}^{+}$ satisfies $\lambda_{k}\geq1$ and
$\lambda_{i}=0$ for $i>k$ with $k<N-m$ then%
\begin{align*}
V^{\left(  0\right)  }\left(  \lambda\right)   &  =q^{\lambda_{k}%
-1}t^{2N-m-k-1}\frac{\left(  1-q^{\lambda_{k}}t^{N-m-k}\right)  \left(
1-q^{\lambda_{k}}t^{N-k+1}\right)  }{\left(  1-q^{\lambda_{k}}t\right)
\left(  1-q^{\lambda_{k}}t^{N-k}\right)  }\\
&  \times\prod\limits_{i=1}^{k-1}\frac{1-q^{\lambda_{i}-\lambda_{k}+1}t^{k-i}%
}{1-q^{\lambda_{i}-\lambda_{k}+1}t^{k-i+1}}V^{\left(  0\right)  }\left(
\lambda^{\prime}\right)  ,
\end{align*}
where $\lambda_{j}^{\prime}=\lambda_{j}$ for $j\neq k$ and $\lambda
_{k}^{\prime}=\lambda_{k}-1$.
\end{proposition}

\begin{proof}
The leading factors are $t^{m}\zeta_{\alpha,E}\left(  1\right)  \left(
y_{N-1}\right)  _{N}=t^{m}q^{\lambda_{k}-1}t^{N-m-k}t^{N-m-1}$.
\end{proof}

\begin{corollary}
Suppose $\lambda$ is as in the Proposition and $\lambda^{\prime\prime}$
satisfies $\lambda_{j}^{\prime\prime}=\lambda_{j}$ for $j\neq k$ and
$\lambda_{k}^{\prime\prime}=0$ then%
\begin{align}
V^{\left(  0\right)  }\left(  \lambda\right)   &  =q^{\binom{\lambda_{k}}{2}%
}t^{\lambda_{k}\left(  2N-m-k-1\right)  }\frac{\left(  qt^{N-m-k};q\right)
_{\lambda_{k}}\left(  qt^{N-k+1};q\right)  _{\lambda_{k}}}{\left(
qt;q\right)  _{\lambda_{k}}\left(  qt^{N-k};q\right)  _{\lambda_{k}}%
}\label{lb2lbb}\\
&  \times\prod\limits_{i=1}^{k-1}\frac{\left(  qt^{k-i};q\right)
_{\lambda_{i}}\left(  qt^{k-i+1};q\right)  _{\lambda_{i}-\lambda_{k}}}{\left(
qt^{k-i};q\right)  _{\lambda_{i}-\lambda_{k}}\left(  qt^{k-i+1};q\right)
_{\lambda_{i}}}V^{\left(  0\right)  }\left(  \lambda^{\prime\prime}\right)
.\nonumber
\end{align}

\end{corollary}

\begin{proof}
This uses%
\begin{equation}
\prod\limits_{l=1}^{\lambda_{k}}\left(  1-q^{\lambda_{i}-l+1}t^{n}\right)
=\frac{\left(  qt^{n};q\right)  _{\lambda_{i}}}{\left(  qt^{n};q\right)
_{\lambda_{i}-\lambda_{k}}} \label{lbijk}%
\end{equation}
with $n=k-i,k-i+1$.
\end{proof}

This formula can now be multiplied out over $k$, starting with $\lambda
=\boldsymbol{0}$, where $M_{\boldsymbol{0},E}\left(  x;\theta\right)
=\tau_{E}\left(  \theta\right)  $.

\begin{theorem}
\label{V0thm}Suppose $\lambda\in\mathcal{N}_{0}^{+}$ then
\begin{equation}
V^{\left(  0\right)  }\left(  \lambda\right)  =q^{\beta\left(  \lambda\right)
}t^{e_{0}\left(  \lambda\right)  }\prod\limits_{k=1}^{N-m-1}\frac{\left(
qt^{N-k+1};q\right)  _{\lambda_{k}}}{\left(  qt^{N-k};q\right)  _{\lambda_{k}%
}}\prod\limits_{1\leq i<j<N-m}\frac{\left(  qt^{j-i+1};q\right)  _{\lambda
_{i}-\lambda_{j}}}{\left(  qt^{j-i};q\right)  _{\lambda_{i}-\lambda_{j}}}
\label{V0lambda}%
\end{equation}
where $\beta\left(  \lambda\right)  :=\sum_{i=1}^{N-m-1}\binom{\lambda_{i}}%
{2}$ and $e_{0}\left(  \lambda\right)  :=\sum_{i=1}^{N-m-1}\lambda_{i}\left(
2N-m-i-1\right)  $.
\end{theorem}

\begin{proof}
For $1\leq k<N-m$ define $\lambda^{\left(  k\right)  }$ by $\lambda
_{i}^{\left(  k\right)  }=\lambda_{i}$ for $1\leq i\leq k$ and $\lambda
_{i}^{\left(  k\right)  }=0$ for $i>k$. Formula (\ref{lb2lbb}) gives the value
of $\rho_{k}:=V^{\left(  0\right)  }\left(  \lambda^{\left(  k\right)
}\right)  /V^{\left(  0\right)  }\left(  \lambda^{\left(  k-1\right)
}\right)  $. For fixed $i\leq k$ the products $\left(  \ast;q\right)
_{\lambda_{i}}$ contribute%
\begin{align*}
&  \frac{\left(  qt^{N-m-i};q\right)  _{\lambda_{i}}\left(  qt^{N-i+1}%
;q\right)  _{\lambda_{i}}}{\left(  qt;q\right)  _{\lambda_{i}}\left(
qt^{N-i};q\right)  _{\lambda_{i}}}\prod\limits_{j=i+1}^{k}\frac{\left(
qt^{j-i};q\right)  _{\lambda_{i}}}{\left(  qt^{j-i+1};q\right)  _{\lambda_{i}%
}}\\
&  =\frac{\left(  qt^{N-m-i};q\right)  _{\lambda_{i}}\left(  qt^{N-i+1}%
;q\right)  _{\lambda_{i}}}{\left(  qt;q\right)  _{\lambda_{i}}\left(
qt^{N-i};q\right)  _{\lambda_{i}}}\frac{\left(  qt;q\right)  _{\lambda_{i}}%
}{\left(  qt^{k-i+1};q\right)  _{\lambda_{i}}}=\frac{\left(  qt^{N-m-i}%
;q\right)  _{\lambda_{i}}\left(  qt^{N-i+1};q\right)  _{\lambda_{i}}}{\left(
qt^{k-i+1};q\right)  _{\lambda_{i}}\left(  qt^{N-i};q\right)  _{\lambda_{i}}}%
\end{align*}
to $\rho_{1}\rho_{2}\cdots\rho_{k}$ (the product telescopes). Each pair
$\left(  i,j\right)  $ with $1\leq i<j\leq k$ contributes $\dfrac{\left(
qt^{j-i+1};q\right)  _{\lambda_{i}-\lambda_{j}}}{\left(  qt^{j-i};q\right)
_{\lambda_{i}-\lambda_{j}}}$. If $\lambda_{k}=0$ then $\rho_{k}=1$ and thus
$k$ can be replaced by $N-m-1$ in the above formulas. The exponents on $q,t$
follow easily from $\rho_{k}$.
\end{proof}

\begin{remark}
Recall the leading term of $M_{\lambda,E}\left(  x;\theta\right)  $, namely
$q^{\beta\left(  \lambda\right)  }t^{e\left(  \lambda,E\right)  }x^{\lambda
}\tau_{E}\left(  \theta\right)  $, where $e\left(  \lambda,E\right)
=\sum_{i=1}^{N}\lambda_{i}\left(  N-i+c\left(  i,E\right)  \right)  $. By
using $c\left(  i,E\right)  =N-m-i$ for $1\leq i<N-m$ one finds that
$e_{0}\left(  \lambda\right)  =\sum_{i=1}^{N-m-1}\lambda_{i}\left(  N+c\left(
i,E\right)  -1\right)  $ so that $e_{0}\left(  \lambda\right)  -e\left(
\lambda,E\right)  =\sum_{i=1}^{N-m-1}\lambda_{i}\left(  i-1\right)  =n\left(
\lambda\right)  $ and $\left(  x^{\left(  0\right)  }\right)  ^{\lambda
}=t^{n\left(  \lambda\right)  }$.
\end{remark}

There is a generalized $\left(  q,t,\lambda\right)  $-Pochhammer symbol%
\[
\left(  a;q,t\right)  _{\lambda}:=\prod\limits_{i=1}^{N}\left(  at^{1-i}%
;q\right)  _{\lambda_{i}}%
\]
and the $k$-product in (\ref{V0lambda}) can be written as $\left(
qt^{N};q,t\right)  _{\lambda}/\left(  qt^{N-1};q,t\right)  _{\lambda}$. In a
later section we will use a hook product formulation which incorporates a
formula for $V^{\left(  0\right)  }\left(  \alpha\right)  $.

\subsection{From $\alpha$ to $\delta$ for type (1)}

To adapt the results for type (0) to type (1) it almost suffices to
interchange $m\leftrightarrow N-m-1$ and replace $t$ by $t^{-1}$. But there
are signs and powers of $t$ , and different formulas involving $\kappa_{-n}$
to worry about. The interchange occurs often enough to get a symbol:

\begin{definition}
\label{defcapxi}Suppose $h\left(  t,m\right)  $ is a function of $t,m$
(possibly also depending on $\lambda$ or $\alpha$) then set $\Xi h\left(
t,m\right)  =h\left(  t^{-1},N-m-1\right)  $
\end{definition}

We will reuse some notations involving $y_{i},p_{i},r_{i}$ and so forth, with
modified definitions (but conceptually the same). In this section we will
prove $M_{\delta,F}\left(  x^{\left(  1\right)  }\right)  =q^{\lambda_{k}%
-1}t^{N-m-k}\dfrac{1-q^{\lambda_{k}}t^{k-N-1}}{1-q^{\lambda_{k}}t^{k-N}%
}M_{\alpha,F}\left(  x^{\left(  1\right)  }\right)  $. Start with $\delta$
(where $\delta_{i}=\lambda_{i}$ for $1\leq i\leq k-1,\delta_{N-m-1}%
=\lambda_{k}$ and $\delta_{i}=0$ otherwise). Let $\beta^{\left(  m\right)
}=\delta$ and $\beta^{\left(  j\right)  }=s_{j-1}\beta^{\left(  j-1\right)  }$
for $m+1\leq j\leq N$ (so that $\beta^{\left(  N\right)  }=\beta$ in
(\ref{lb2a2b})). Abbreviate $z=\zeta_{\delta,F}\left(  m\right)
=\zeta_{\lambda,F}\left(  k\right)  =q^{\lambda_{k}}t^{k-1-m}$. If $m\leq i<N$
then $\zeta_{\beta^{\left(  i+1\right)  },F}\left(  i+1\right)  =z,\zeta
_{\beta^{\left(  i+1\right)  },F}\left(  i\right)  =t^{N-i-1}$. Set
$p_{i}\left(  x\right)  =M_{\beta^{\left(  i\right)  },F}\left(
x;\theta\right)  $ for $m\leq i\leq N$, then%
\begin{align}
p_{i}\left(  x\right)   &  =\left(  \boldsymbol{T}_{i-1}+\frac{1-t}%
{1-zt^{i+1-N}}\right)  p_{i+1}\left(  x\right) \label{pi2pi1(1)}\\
&  =\left(  b\left(  x;i\right)  +\frac{1-t}{1-zt^{i+1-N}}\right)
p_{i+1}\left(  x\right)  +\left(  T_{i}-b\left(  x;i\right)  \right)
p_{i+1}\left(  xs_{i}\right)  .\nonumber
\end{align}

These are analogs of the type (0) definitions, with $m+1\leq i\leq N-1$:%
\begin{align*}
y_{m}  &  =x^{\left(  1\right)  },y_{m+1}=x^{\left(  1\right)  }s_{m+1}%
,y_{i}=y_{i-1}s_{i},\\
v_{m}  &  =x^{\left(  1\right)  }s_{m},v_{i}=v_{i-1}s_{i},\\
P_{i}  &  :=\left(  T_{m+1}+1\right)  \left(  T_{m+2}-\kappa_{-2}\right)
\cdots\left(  T_{i}-\kappa_{m-i}\right)  .
\end{align*}
In more detail
\begin{align}
y_{i-1}  &  =\left(  \ldots,\overset{m}{t^{1-m}},t^{-m-1},\ldots,\overset
{i}{t^{-m}},t^{-i},\ldots,t^{1-N}\right)  ,b\left(  y_{i-1};i\right)
=\kappa_{m-i},\label{yivi1}\\
v_{i-1}  &  =\left(  \ldots,\overset{m}{t^{-m}},t^{-m-1},\ldots,\overset
{i}{t^{1-m}},t^{-i},\ldots,t^{1-N}\right)  ,b\left(  v_{i-1};i\right)
=\kappa_{m-i-1}..\nonumber
\end{align}

Recall $\zeta_{\delta^{\left(  i\right)  },F}\left(  i\right)  =z,\zeta
_{\delta^{\left(  i\right)  },F}\left(  i+1\right)  =t^{N-i-1}$ for $m\leq
i<N.($The proof of the following is mostly the same as that for Proposition
\ref{P0pTp} except for signs and powers of $t$.)

\begin{proposition}
For $m+1\leq i\leq N$%
\begin{align}
p_{m}\left(  x^{\left(  1\right)  }\right)   &  =-t\frac{1-zt^{m-N}%
}{1-zt^{m+1-N}}P_{i-1}p_{i}\left(  y_{i-1}\right) \label{pmx(zero)}\\
&  +\left(  T_{m}+1\right)  \left(  T_{m+1}-\kappa_{-2}\right)  \cdots\left(
T_{i-1}-\kappa_{m-i}\right)  p_{i}\left(  v_{i-1}\right)  .\nonumber
\end{align}

\end{proposition}

\begin{proof}
The transformation from $p_{i+1}$ to $p_{i}$ is in (\ref{pi2pi1(1)}).
Specialize to $i=m$ and $x=x^{\left(  1\right)  }$ so that $b\left(
x;i\right)  =-1$, $x^{\left(  1\right)  }s_{m}=v_{m}$, and%
\[
p_{m}\left(  x^{\left(  1\right)  }\right)  =-t\frac{1-zt^{m-N}}{1-zt^{m+1-N}%
}p_{m+1}\left(  y_{m}\right)  +\left(  T_{m}+1\right)  p_{m+1}\left(
v_{mi}\right)  .
\]
The values from (\ref{yivi1}) are $b\left(  y_{i-1};i\right)  =\kappa
_{m-i},b\left(  v_{i-1};i\right)  =\kappa_{m-i-1}$, $\zeta_{\delta^{\left(
i\right)  },F}\left(  i\right)  =z$. Thus%
\begin{align*}
p_{i}\left(  y_{i-1}\right)   &  =\left(  \frac{1-t}{1-zt^{i+1-N}}%
+\kappa_{m-i}\right)  p_{i+1}\left(  y_{i-1}\right)  +\left(  T_{i}%
-\kappa_{m-i}\right)  p_{i+1}\left(  y_{i-1}s_{i}\right) \\
p_{i}\left(  v_{i-1}\right)   &  =\left(  \frac{1-t}{1-zt^{i+1-N}}%
+\kappa_{m-i-1}\right)  p_{i+1}\left(  v_{i-1}\right)  +\left(  T_{i}%
-\kappa_{m-i-1}\right)  p_{i+1}\left(  v_{i-1}s_{i}\right)  ,
\end{align*}
then%
\begin{align*}
\frac{1-t}{1-zt^{i+1-N}}+\kappa_{m-i}  &  =-\frac{1-zt^{m-N+1}}{1-zt^{i+1-N}%
}\frac{t^{i-m}}{\left[  i-m\right]  _{t}},\\
\frac{1-t}{1-zt^{i+1-N}}+\kappa_{m-i-1}  &  =-\frac{1-zt^{m-N}}{1-zt^{i+1-N}%
}\frac{t^{i+1-m}}{\left[  i+1-m\right]  _{t}}%
\end{align*}
From the spectral vector of $p_{i+1}$ it follows that $\left(  \boldsymbol{T}%
_{j}-t\right)  p_{i+1}=0$ for $m\leq j<i$ and $\left(  T_{j}-b\left(
x;j\right)  \right)  p_{i+1}\left(  xs_{j}\right)  =\left(  t-b\left(
x;j\right)  \right)  p_{i+1}\left(  x\right)  $. Thus $\left(  T_{i-1}%
-\kappa_{m+1-i}\right)  p_{i+1}\left(  y_{i-1}\right)  =\frac{\left[
i-m\right]  _{t}}{\left[  i-m-1\right]  _{t}}p_{i+1}\left(  y_{i-2}\right)  $
and
\begin{align*}
\left(  T_{m+1}+1\right)  \cdots\left(  T_{i-1}-\kappa_{m+1-i}\right)
p_{i+1}\left(  y_{i-1}\right)   &  =\left[  i-m\right]  _{t}p_{i+1}\left(
y_{m}\right) \\
\left(  T_{m}+1\right)  \left(  T_{m+1}-\kappa_{-2}\right)  \cdots\left(
T_{i-1}-\kappa_{m-i}\right)  p_{i+1}\left(  v_{i-1}\right)   &  =\left[
i+1-m\right]  _{t}p_{i+1}\left(  x^{\left(  1\right)  }\right)  .
\end{align*}
Then $p_{i+1}\left(  y_{m}\right)  $ appears in the expression for
$p_{m}\left(  x^{\left(  1\right)  }\right)  $ with factor%
\[
\left(  -t\frac{1-zt^{m-N}}{1-zt^{m+1-N}}\right)  \left(  -\frac{1-zt^{m-N+1}%
}{1-zt^{i+1-N}}\frac{t^{i-m}}{\left[  i-m\right]  _{t}}\right)  \left[
i-m\right]  _{t}=t^{i+1-m}\frac{1-zt^{m-N}}{1-zt^{i+1-N}}%
\]
and $p_{i+1}\left(  x^{\left(  1\right)  }\right)  $ with factor
$-\dfrac{1-zt^{m-N}}{1-zt^{i+1-N}}\dfrac{t^{i+1-m}}{\left[  i+1-m\right]
_{t}}\left[  i+1-m\right]  _{t}$ and the two cancel out ($y_{m}=x^{\left(
1\right)  }$). This proves the inductive step.
\end{proof}

\begin{proposition}
$M_{\delta,F}\left(  x^{\left(  1\right)  }\right)  =-t\dfrac{1-zt^{m-N}%
}{1-zt^{m+1-N}}P_{N-1}p_{N}\left(  y_{N-1}\right)  $.
\end{proposition}

\begin{proof}
Set $i=N$ in (\ref{pmx(zero)}). Claim $\left(  T_{m}+1\right)  \left(
T_{m+1}-\kappa_{-2}\right)  \cdots\left(  T_{N-1}-\kappa_{m-N}\right)
p_{N}\left(  v_{N-1}\right)  =0$. From $\left(  v_{N-1}\right)  _{i}=t^{-i}$
and $\zeta_{\beta,F}\left(  i\right)  =t^{N-i-1}$ for $m\leq i\leq N-1$ it
follows that $\left(  T_{i}-t\right)  p_{N}\left(  v_{N-1}\right)  =0$ for
$m\leq i\leq N-2$. This implies $p_{N}\left(  v_{N-1}\right)  =cM\left(
\theta_{1}\theta_{2}\cdots\theta_{m-1}\theta_{N}\right)  $ for some constant (
similarly to the argument in Proposition \ref{MzM0} $h=T_{m}T_{m+1}\cdots
T_{N-1}M\left(  \theta_{1}\cdots\theta_{m-1}\theta_{N}\right)  $ satisfies
$\left(  T_{i}-t\right)  h=0$ for $m+1\leq i\leq N-1$ implying $h=c^{\prime
}\tau_{F}$). Let $g=\theta_{1}\theta_{2}\cdots\theta_{m-1}$ and $f_{0}%
=g\theta_{N}$ then define $f_{i}=\left(  T_{N-i}-\kappa_{m-N+i-1}\right)
f_{i-1}$ for $1\leq i\leq N-m$. Use induction to show
\[
f_{i}=t^{i}g\theta_{N-i}+\frac{t^{N-m}}{\left[  N-m+1-i\right]  _{t}}%
\sum_{j=0}^{i-1}g\theta_{N-j}.
\]
The start is%
\begin{align*}
f_{1}  &  =g\theta_{N}=\left(  T_{N-1}+\frac{1}{\left[  N-m\right]  _{t}%
}\right)  g\theta_{N}=tg\theta_{N-1}+\left(  t-1+\frac{1}{\left[  N-m\right]
_{t}}\right)  g\theta_{N}\\
&  =tg\theta_{N-1}+\frac{t^{N-m}}{\left[  N-m\right]  _{t}}g\theta_{N}.
\end{align*}
Assume the formula is true for some $i<N-m$ then%
\begin{align*}
f_{i+1}  &  =\left(  T_{N-i-1}-\kappa_{m-N+i}\right)  f_{i}=t^{i+1}%
g\theta_{N-i-1}+t^{i}\left(  t-1+\frac{1}{\left[  N-m-i\right]  _{t}}\right)
g\theta_{N-i}\\
&  +\frac{t^{N-m}}{\left[  N-m+1-i\right]  _{t}}\left\{  t+\frac{1}{\left[
N-m-i\right]  _{t}}\right\}  \sum_{j=0}^{i-1}g\theta_{N-j}\\
&  =t^{i+1}g\theta_{N-i-1}+\frac{t^{N-m}}{\left[  N-m-i\right]  _{t}}%
\sum_{j=0}^{i}g\theta_{N-j}%
\end{align*}
(because $t+\frac{1}{\left[  n\right]  _{t}}=\frac{\left[  n+1\right]  _{t}%
}{\left[  n\right]  t}$). Thus $f_{N-m}=t^{N-m}g\sum_{j=0}^{N-m}\theta
_{N-j}=\left(  -1\right)  ^{m-1}t^{N-m}M\left(  g\right)  $ and $M\left(
f_{N-m}\right)  =0.$
\end{proof}

Next we consider the transition from $\alpha$ to $\beta$ (see \ref{lb2a2b})
with the affine step $M_{\beta,F}\left(  x;\theta\right)  =x_{N}%
\boldsymbol{w}M_{\alpha,F}\left(  x;\theta\right)  $ and as before the
calculation is based on the formula%
\[
M_{\beta,F}\left(  x;\theta\right)  =x_{N}t^{N-1}\left(  \boldsymbol{T}%
_{N-1}^{-1}\cdots\boldsymbol{T}_{2}^{-1}\boldsymbol{T}_{1}^{-1}\xi
_{1}M_{\alpha,F}\right)  \left(  x;\theta\right)
\]
where $\xi_{1}M_{\alpha,F}=\zeta_{\alpha,F}\left(  1\right)  M_{\alpha
,F}=q^{\lambda_{k}-1}t^{k-m-1}M_{\alpha,F}$. From the previous formula we see
that we need to evaluate $P_{N-1}M_{\beta,F}\left(  y_{N-1}\right)  $. Since
$\zeta_{\alpha,F}\left(  i\right)  =t^{N-i}$ for $m+1\leq i\leq N$ it follows
that $\left(  \boldsymbol{T}_{i}-t\right)  M_{\alpha,F}=0$ for $m+1\leq i<N$.

\begin{definition}
Let $r_{0}=M_{\alpha,F}$ and $r_{i}=t\boldsymbol{T}_{i}^{-1}r_{i-1}$ for
$1\leq i<N$.
\end{definition}

\begin{proposition}
Suppose $1\leq i\leq m-1$ then $r_{i}\left(  x^{\left(  1\right)  }\right)
=\left(  -t\right)  ^{i}r_{0}\left(  x^{\left(  1\right)  }\right)  =\left(
-t\right)  ^{i}M_{\alpha,F}\left(  x^{\left(  1\right)  }\right)  .$
\end{proposition}

\begin{proof}
From $\left(  \boldsymbol{T}_{i}-t\right)  M_{\alpha,F}=0$ for $m+1\leq i<N$
it follows that $\left(  \boldsymbol{T}_{i}-t\right)  r_{j}=0$ if $j<m$ . By
(\ref{r2r}) $r_{\ell}\left(  x^{\left(  1\right)  }\right)  =-tr_{\ell
-1}\left(  x^{\left(  1\right)  }\right)  +\left(  T_{\ell}+1\right)
r_{\ell-1}\left(  x^{\left(  1\right)  }s_{\ell}\right)  $ . Suppose
$\ell<N-m-1$ then $x^{\left(  0\right)  }s_{\ell}$ satisfies $\left(
x^{\left(  1\right)  }s_{\ell}\right)  _{i}=t^{-i-}$ for $i\geq m+1$ so that
$b\left(  x^{\left(  0\right)  }s_{\ell};i\right)  =-1$ and $\left(
T_{i}-t\right)  r_{\ell-1}\left(  x^{\left(  1\right)  }s_{\ell}\right)  =0$,
$r_{\ell-1}\left(  x^{\left(  1\right)  }s_{\ell}\right)  $ is a multiple of
$\tau_{F}$ and $\left(  T_{\ell}+1t\right)  r_{\ell-1}\left(  x^{\left(
1\right)  }s_{\ell}\right)  =0$. Thus $r_{\ell}\left(  x^{\left(  0\right)
}\right)  =-tr_{\ell-1}\left(  x^{\left(  0\right)  }\right)  $ and this holds
for $1\leq\ell\leq m-1$.
\end{proof}

Similarly to the type (0) computations let
\begin{align*}
\widetilde{P}_{N-j}  &  :=\left(  T_{N-1}+1\right)  \left(  T_{N-2}%
-\kappa_{-2}\right)  \cdots\left(  T_{N-j}-\kappa_{-j}\right) \\
\widetilde{y}_{N-j-1}  &  :=y_{N-j-1}s_{N-1}s_{N-2}\cdots s_{N-j}.
\end{align*}

\begin{lemma}
\label{TTfj1}Suppose $\boldsymbol{T}_{i}f=tf$ for $N-j\leq i<N$ and
$u_{i}=ct^{1-i}$ for $N-j+1\leq i\leq N$ then%
\[
\left(  T_{N-1}+1\right)  \left(  T_{N-2}-\kappa_{-2}\right)  \cdots\left(
T_{N-j}-\kappa_{-j}\right)  f\left(  us_{N-1}s_{N-2}\cdots s_{N-j}\right)
=\left[  j+1\right]  _{t}f\left(  u\right)  .
\]

\end{lemma}

\begin{proof}
Let $\widetilde{u}^{\left(  k\right)  }=us_{N-1}\cdots s_{N-k}$ then $\left(
\widetilde{u}^{\left(  k-1\right)  }\right)  _{N-k+1}=t^{1-N},\left(
\widetilde{u}^{\left(  k-1\right)  }\right)  _{N-k}=t^{k-N}$ and $b\left(
\widetilde{u}^{\left(  k-1\right)  };N-k\right)  =\kappa_{1-k}$. Thus%
\begin{align*}
\left(  T_{N-k}-\kappa_{1-k}\right)  f\left(  \widetilde{u}^{\left(  k\right)
}\right)   &  =\left(  T_{N-k}-\kappa_{1-k}\right)  f\left(  \widetilde
{u}^{\left(  k-1\right)  }s_{N-k}\right) \\
&  =\left(  t-\kappa_{1-k}\right)  f\left(  \widetilde{u}^{\left(  k-1\right)
}\right)  =\dfrac{\left[  k\right]  _{t}}{\left[  k-1\right]  _{t}}.
\end{align*}
Repeated application of this formula shows%
\begin{align*}
\left(  T_{N-1}+1\right)  \cdots\left(  T_{N-j}-\kappa_{-j}\right)  f\left(
us_{N-1}s_{N-2}\cdots s_{N-j}\right)   &  =\frac{1}{\left[  2\right]  _{t}%
}\frac{\left[  2\right]  _{t}}{\left[  3\right]  _{t}}\cdots\frac{\left[
j+1\right]  _{t}}{\left[  j\right]  _{t}}f\left(  u\right) \\
&  =\left[  j+1\right]  _{t}f\left(  u\right)  .
\end{align*}

\end{proof}

\begin{proposition}
For $2\leq j\leq N-m$%
\begin{align}
P_{N-1}r_{N-1}\left(  y_{N-1}\right)   &  =t^{j-1}\frac{\left[  N-m\right]
_{t}\left[  N-m-j\right]  _{t}}{\left[  N-m-1\right]  _{t}\left[
N-m+1-j\right]  _{t}}P_{N-j}r_{N-j}\left(  y_{N-j}\right) \label{PrPt1a}\\
&  +\frac{1}{\left[  N-m+1-j\right]  _{t}}\widetilde{P}_{N-j+2}P_{N-j}\left(
T_{N-j+1}-\kappa_{m+1-N}\right)  r_{N-j}\left(  \widetilde{y}_{N-j}\right)  .
\label{PrPt1b}%
\end{align}

\end{proposition}

\begin{proof}
Proceed by induction. By (\ref{TbfTb})%
\begin{align*}
P_{N-1}r_{N-1}\left(  y_{N-1}\right)   &  =P_{N-2}\left(  T_{N-1}%
-\kappa_{m-N+1}\right)  r_{N-1}\left(  y_{N-2}s_{N-1}\right) \\
&  =\dfrac{t\left[  N-m\right]  _{t}\left[  N-m-2\right]  _{t}}{\left[
N-m-1\right]  _{t}^{2}}P_{N-2}r_{N-2}\left(  y_{N-2}\right) \\
&  +\frac{1}{\left[  N-m-1\right]  _{t}}P_{N-2}\left(  T_{N-1}-\kappa
_{m+1-N}\right)  r_{N-2}\left(  y_{N-2}s_{N-1}\right)
\end{align*}
and $y_{N-1}=y_{N-2}s_{N-1}=\widetilde{y}_{N-2}$. Thus the formula is valid
for $j=2$ (with $\widetilde{P}_{N}=1$). Suppose it holds for some $j\leq
N-m-2$, then $b\left(  y_{N-1-j};N-j\right)  =\kappa_{j+m-N}$ and%
\begin{align*}
P_{N-j}r_{N-j}\left(  y_{N-j}\right)   &  =P_{N-j-1}\left(  T_{N-j}%
-\kappa_{j+m-N}\right)  r_{N-j}\left(  y_{N-j-1}s_{N-j}\right) \\
&  =\dfrac{t\left[  N-m+1-j\right]  _{t}\left[  N-m-1-j\right]  _{t}}{\left[
N-m-j\right]  _{t}^{2}}P_{N-j-1}r_{N-j-1}\left(  y_{N-j-1}\right) \\
&  +\frac{1}{\left[  N-m-j\right]  _{t}}P_{N-j-1}\left(  T_{N-j}%
-\kappa_{j+m-N}\right)  r_{N-j-1}\left(  y_{N-j}\right)  .
\end{align*}
Combine with formula (\ref{PrPt1a}) to obtain%
\begin{align}
&  t^{j}\frac{\left[  N-m\right]  _{t}\left[  N-m-1-j\right]  _{t}}{\left[
N-m-1\right]  _{t}\left[  N-m-j\right]  _{t}}P_{N-j-1}r_{N-j-1}\left(
y_{N-j-1}\right) \label{BBB}\\
&  +\frac{t^{j-1}\left[  N-m\right]  _{t}}{\left[  N-m-1\right]  _{t}\left[
N-m+1-j\right]  _{t}}P_{N-j}r_{N-j-1}\left(  y_{N-j}\right)  .
\end{align}
For the second line (\ref{PrPt1b}) $b\left(  \widetilde{y}_{N-j};N-j\right)
=\kappa_{1-j}$ thus%
\[
r_{N-j}\left(  \widetilde{y}_{N-j}\right)  =-\frac{t^{j-1}}{\left[
j-1\right]  _{t}}r_{N-j-1}\left(  \widetilde{y}_{N-j}\right)  +\left(
T_{N-j}-\kappa_{1-j}\right)  r_{N-j-1}\left(  \widetilde{y}_{N-j}%
s_{N-j}\right)  .
\]
The first part leads to%
\begin{align*}
&  -\frac{t^{j-1}}{\left[  N-m+1-j\right]  _{t}\left[  j-1\right]  _{t}%
}\widetilde{P}_{N-j+2}P_{N-j}\left(  T_{N-j+1}-\kappa_{m+1-N}\right)
r_{N-j-1}\left(  \widetilde{y}_{N-j}\right) \\
&  =-\frac{\left[  N-m\right]  _{t}}{\left[  N-m-1\right]  _{t}}\frac{t^{j-1}%
}{\left[  N-m+1-j\right]  _{t}\left[  j-1\right]  _{t}}\widetilde{P}%
_{N-j+2}P_{N-j}r_{N-j-1}\left(  \widetilde{y}_{N-j}s_{N-j+1}\right)
\end{align*}
because $b\left(  \widetilde{y}_{N-j}s_{N-j+1};N-j+1\right)  =\kappa_{m+1-N}$
and $\left(  \boldsymbol{T}_{N-j+1}-t\right)  r_{N-j-1}=0$. Then%
\begin{align*}
\widetilde{P}_{N-j+2}P_{N-j}r_{N-j-1}\left(  \widetilde{y}_{N-j}%
s_{N-j+1}\right)   &  =P_{N-j}\widetilde{P}_{N-j+2}r_{N-j-1}\left(
y_{N-j}s_{N-1}\cdots s_{N-j+2}\right) \\
&  =\left[  j-1\right]  _{t}P_{N-j}r_{N-j-1}\left(  y_{N-j}\right)
\end{align*}
by Lemma \ref{TTfj1}, so combine to obtain $\dfrac{\left[  N-m\right]  _{t}%
}{\left[  N-m-1\right]  _{t}}\dfrac{\left(  -1\right)  t^{j-1}}{\left[
N-m+1-j\right]  _{t}}P_{N-j}r_{N-j-1}\left(  y_{N-j}\right)  $ which cancels
the second term in (\ref{BBB}). The second part gives (using the braid
relations in (\ref{Tbraids}))%
\begin{align*}
&  \frac{1}{\left[  N-m+1-j\right]  _{t}}\widetilde{P}_{N-j+2}P_{N-j}\left(
T_{N-j+1}-\kappa_{m+1-N}\right)  \left(  T_{N-j}-\kappa_{1-j}\right)
r_{N-j-1}\left(  \widetilde{y}_{N-j}s_{N-j}\right) \\
&  =\frac{1}{\left[  N-m+1-j\right]  _{t}}\widetilde{P}_{N-j+2}P_{N-j-1}\\
&  \times\left(  T_{N-j}-\kappa_{m+j-N}\right)  \left(  T_{N-j+1}%
-\kappa_{m+1-N}\right)  \left(  T_{N-j}-\kappa_{1-j}\right)  r_{N-j-1}\left(
\widetilde{y}_{N-j}s_{N-j}\right) \\
&  =\frac{1}{\left[  N-m+1-j\right]  _{t}}\widetilde{P}_{N-j+2}P_{N-j-1}\\
&  \times\left(  T_{N-j+1}-\kappa_{1-j}\right)  \left(  T_{N-j}-\kappa
_{m+1-N}\right)  \left(  T_{N-j+1}-\kappa_{m+j-N}\right)  r_{N-j-1}\left(
\widetilde{y}_{N-j}s_{N-j}\right) \\
&  =\frac{\left[  N-m+1-j\right]  _{t}}{\left[  N-m+1-j\right]  _{t}\left[
N-m-j\right]  _{t}}\widetilde{P}_{N-j+2}P_{N-j-1}\\
&  \times\left(  T_{N-j+1}-\kappa_{1-j}\right)  \left(  T_{N-j}-\kappa
_{m+1-N}\right)  r_{N-j-1}\left(  \widetilde{y}_{N-j}s_{N-j}s_{N-j+1}\right)
\\
&  =\frac{1}{\left[  N-m-j\right]  _{t}}\widetilde{P}_{N-j+1}P_{N-j-1}\left(
T_{N-j}-\kappa_{m+1-N}\right)  r_{N-j-1}\left(  \widetilde{y}_{N-j-1}\right)
\end{align*}
by (\ref{yi2y}) and $b\left(  \widetilde{y}_{N-j-1};N-j+1\right)
=\kappa_{m+j-N}.$
\end{proof}

\begin{proposition}
$P_{N-1}r_{N-1}\left(  y_{N-1}\right)  =\left(  -1\right)  ^{m}t^{N-1}%
M_{\beta,F}\left(  x^{\left(  1\right)  };\theta\right)  $.
\end{proposition}

\begin{proof}
Set $j=N-m$ in (\ref{PrPt1a})%
\[
P_{N-1}r_{N-1}\left(  y_{N-1}\right)  =\widetilde{P}_{m+2}\left(
T_{m+1}-\kappa_{m+1-N}\right)  r_{m}\left(  \widetilde{y}_{m}\right)
=\widetilde{P}_{m+1}r_{m}\left(  \widetilde{y}_{m}\right)
\]
and $b\left(  \widetilde{y}_{m};m\right)  =\kappa_{m-N}$ (note $\widetilde
{y}_{m}=x^{\left(  1\right)  }s_{N-1}s_{N-2}\cdots s_{m+1}$) thus%
\[
\widetilde{P}_{m+1}r_{m}\left(  \widetilde{y}_{m}\right)  =-\frac{t^{N-m}%
}{\left[  N-m\right]  _{t}}\widetilde{P}_{m+1}r_{m-1}\left(  \widetilde{y}%
_{m}\right)  +\widetilde{P}_{m+1}\left(  T_{m}-\kappa_{m-N}\right)
r_{m-1}\left(  \widetilde{y}_{m}s_{m}\right)  .
\]
Now $\widetilde{y}_{m}=\left(  \ldots,\overset{\left(  m\right)  }{t^{1-m}%
},\overset{\left(  m+1\right)  }{t^{1-N}},t^{-m},\ldots,t^{2-N}\right)  $ thus
$\widetilde{y}_{m}$ satisfies the hypothesis of Lemma \ref{TTfj1} with
$j=N-m-1$ and%
\begin{align*}
&  \left(  T_{N-1}+1\right)  \left(  T_{N-2}-\kappa_{-2}\right)  \cdots\left(
T_{m+1}-\kappa_{m+1-N}\right)  r_{m-1}\left(  x^{\left(  1\right)  }%
s_{N-1}s_{N-2}\cdots s_{m+1}\right) \\
&  =\left[  N-m\right]  _{t}r_{m-1}\left(  x^{\left(  1\right)  }\right)  .
\end{align*}
Since $\widetilde{y}_{m}s_{m}=\left(  \ldots,t^{1-N},\overset{\left(
m+1\right)  }{t^{1-m}},t^{-m},\ldots t^{2-N}\right)  $ and $\left(
\boldsymbol{T}_{i}-t\right)  r_{m-1}=0$ for $m+1\leq i<N$ it follows that
$\left(  T_{i}-t\right)  r_{m-1}\left(  \widetilde{y}_{m}s_{m}\right)  =0$ for
the same $i$ values and hence $r_{m-1}\left(  \widetilde{y}_{m}s_{m}\right)
=c\tau_{F}$ (with $F=\left\{  1,2,\ldots,m\right\}  $ because $m+1,m+2,\ldots
,N$ lie in the same row of $Y_{F}$). Take $\tau_{F}=M\left(  \theta_{1}%
\cdots\theta_{m}\right)  $ and $g=\theta_{1}\cdots\theta_{m-1}$ then%
\begin{align*}
\left(  T_{m}-\kappa_{m-N}\right)  \left(  g\theta_{m}\right)   &
=g\theta_{m+1}+\frac{1}{\left[  N-m\right]  _{t}}g\theta_{m},\\
\left(  T_{m+1}-\kappa_{m+1-N}\right)  \left(  T_{m}-\kappa_{m-N}\right)
\left(  g\theta_{m}\right)   &  =g\theta_{m+2}+\frac{1}{\left[  N-m-1\right]
_{t}}g\theta_{m+1}\\
&  +\left(  t+\frac{1}{\left[  N-m-1\right]  _{t}}\right)  \frac{1}{\left[
N-m\right]  _{t}}g\theta_{m}\\
&  =g\theta_{m+2}+\frac{1}{\left[  N-m-1\right]  _{t}}g\left(  \theta
_{m+1}+\theta_{m}\right)  ,
\end{align*}
because $t+\frac{1}{\left[  j-1\right]  _{t}}=\frac{\left[  j\right]  _{t}%
}{\left[  j-1\right]  _{t}}$. Continue this process to obtain%
\[
\left(  T_{N-1}-\kappa_{-1}\right)  \cdots\left(  T_{m}-\kappa_{m-N}\right)
\left(  g\theta_{m}\right)  =g\left(  \theta_{N}+\cdots+\theta_{m}\right)
=\left(  -1\right)  ^{m-1}M\left(  g\right)
\]
thus $\left(  T_{N-1}-\kappa_{-1}\right)  \cdots\left(  T_{m}-\kappa
_{m-N}\right)  r_{m-1}\left(  \widetilde{y}_{m}s_{m}\right)  =0$ because
$M^{2}=0$. Thus $P_{N-1}r_{N-1}\left(  y_{N-1}\right)  =-t^{N-m}r_{m-1}\left(
x^{\left(  1\right)  }\right)  =\left(  -1\right)  ^{m}t^{N-1}r_{0}\left(
x^{\left(  1\right)  }\right)  $.
\end{proof}

\subsection{Evaluation formula for type (1)}

Recall the intermediate steps:%
\begin{align*}
V^{\left(  1\right)  }\left(  \alpha\right)   &  =\left(  -t\right)
^{1-k}\prod\limits_{i=1}^{k-1}\frac{1-q^{\lambda_{i}-\lambda_{k}+1}t^{i-i}%
}{1-q^{\lambda_{i}-\lambda_{k}+1}t^{i-k-1}}V^{\left(  1\right)  }\left(
\lambda^{\prime}\right) \\
P_{N-1}r_{N-1}\left(  y_{N-1}\right)   &  =\left(  -1\right)  ^{m}%
t^{N-1}M_{\alpha,F}\left(  x^{\left(  1\right)  };\theta\right) \\
M_{\beta,F}\left(  y_{N-1};\theta\right)   &  =\zeta_{\alpha,F}\left(
1\right)  \left(  y_{N-1}\right)  _{N}~r_{N-1}\left(  y_{N-1}\right) \\
M_{\delta,F}\left(  x^{\left(  1\right)  };\theta\right)   &  =-t\dfrac
{1-q^{\lambda_{k}}t^{k-N-1}}{1-q^{\lambda_{k}}t^{k-N}}P_{N-1}M_{\beta
,F}\left(  y_{N-1};\theta\right) \\
V^{\left(  1\right)  }\left(  \lambda\right)   &  =\left(  -t\right)
^{m-k}\frac{1-q^{\lambda_{k}}t^{k-m-1}}{1-q^{\lambda_{k}}t^{-1}}V^{\left(
1\right)  }\left(  \delta\right)  .
\end{align*}

\begin{proposition}
Suppose $\lambda\in\mathcal{N}_{1}^{+}$ satisfies $\lambda_{k}\geq1$ and
$\lambda_{i}=0$ for $i>k$ with $k\leq m$ then%
\begin{align*}
V^{\left(  1\right)  }\left(  \lambda\right)   &  =q^{\lambda_{k}-1}%
t^{N-m-k}\frac{\left(  1-q^{\lambda_{k}}t^{k-m-1}\right)  \left(
1-q^{\lambda_{k}}t^{k-N-1}\right)  }{\left(  1-q^{\lambda_{k}}t^{-1}\right)
\left(  1-q^{\lambda_{k}}t^{k-N}\right)  }\\
&  \times\prod\limits_{i=1}^{k-1}\frac{1-q^{\lambda_{i}-\lambda_{k}+1}t^{i-k}%
}{1-q^{\lambda_{i}-\lambda_{k}+1}t^{i-k-1}}V^{\left(  1\right)  }\left(
\lambda^{\prime}\right)  ,
\end{align*}
where $\lambda_{j}^{\prime}=\lambda_{j}$ for $j\neq k$ and $\lambda
_{k}^{\prime}=\lambda_{k}-1$.
\end{proposition}

\begin{proof}
The leading factors are $t^{N+m-2k+1}\zeta_{\alpha,F}\left(  1\right)  \left(
y_{N-1}\right)  _{N}=t^{N-m-k}q^{\lambda_{k}-1}$, since $\zeta_{\alpha
,F}\left(  1\right)  =q^{\lambda_{k}-1}t^{k-1-m}$ and $\left(  y_{N-1}\right)
_{N}=t^{-m}$.
\end{proof}

\begin{corollary}
Suppose $\lambda$ is as in the Proposition and $\lambda^{\prime\prime}$
satisfies $\lambda_{j}^{\prime\prime}=\lambda_{j}$ for $j\neq k$ and
$\lambda_{k}^{\prime\prime}=0$ then%
\begin{align*}
V^{\left(  1\right)  }\left(  \lambda\right)   &  =q^{\binom{\lambda_{k}}{2}%
}t^{\lambda_{k}\left(  N-m-k\right)  }\frac{\left(  qt^{k-m-1};q\right)
_{\lambda_{k}}\left(  qt^{k-N-1};q\right)  _{\lambda_{k}}}{\left(
qt^{-1};q\right)  _{\lambda_{k}}\left(  qt^{k-N};q\right)  _{\lambda_{k}}}\\
&  \times\prod\limits_{i=1}^{k-1}\frac{\left(  qt^{i-k};q\right)
_{\lambda_{i}}\left(  qt^{i-k-1};q\right)  _{\lambda_{i}-\lambda_{k}}}{\left(
qt^{i-k};q\right)  _{\lambda_{i}-\lambda_{k}}\left(  qt^{i-k-1};q\right)
_{\lambda_{i}}}V^{\left(  1\right)  }\left(  \lambda^{\prime\prime}\right)  .
\end{align*}

\end{corollary}

\begin{proof}
This uses formula (\ref{lbijk}).
\end{proof}

This formula can now be multiplied out over $k$, starting with $\lambda
=\boldsymbol{0}$, where $M_{\boldsymbol{0},F}\left(  x;\theta\right)
=\tau_{F}\left(  \theta\right)  $.

\begin{theorem}
\label{V1lambda}Suppose $\lambda\in\mathcal{N}_{1}^{+}$ then
\[
V^{\left(  1\right)  }\left(  \lambda\right)  =q^{\beta\left(  \lambda\right)
}t^{e_{1}\left(  \lambda\right)  }\prod\limits_{k=1}^{m}\frac{\left(
qt^{k-N-1};q\right)  _{\lambda_{k}}}{\left(  qt^{k-N};q\right)  _{\lambda_{k}%
}}\prod\limits_{1\leq i<j\leq m}\frac{\left(  qt^{i-1-1};q\right)
_{\lambda_{i}-\lambda_{j}}}{\left(  qt^{i-j};q\right)  _{\lambda_{i}%
-\lambda_{j}}}%
\]
where $\beta\left(  \lambda\right)  :=\sum_{i=1}^{N-m-1}\binom{\lambda_{i}}%
{2}$ and $e_{1}\left(  \lambda\right)  :=\sum_{i=1}^{m}\lambda_{i}\left(
N-m-i\right)  $.
\end{theorem}

\begin{proof}
This is the same argument used in Theorem \ref{V0thm} by the application of
$\Xi$ .
\end{proof}

\begin{remark}
Recall the leading term of $M_{\lambda,F}\left(  x;\theta\right)  $, namely
$q^{\beta\left(  \lambda\right)  }t^{e\left(  \lambda,F\right)  }x^{\lambda
}\tau_{F}\left(  \theta\right)  $, where $e\left(  \lambda,F\right)
=\sum_{i=1}^{N}\lambda_{i}\left(  N-i+c\left(  i,F\right)  \right)  $. By
using $c\left(  i,F\right)  =i-m-1$ for $1\leq i\leq m$ one finds that
$e_{1}\left(  \lambda\right)  =\sum_{i=1}^{m}\lambda_{i}\left(  N+c\left(
i,F\right)  -2i+1\right)  $ so that $e_{1}\left(  \lambda\right)  -e\left(
\lambda,F\right)  =-\sum_{i=1}^{m}\lambda_{i}\left(  i-1\right)  =-n\left(
\lambda\right)  $ and $\left(  x^{\left(  1\right)  }\right)  ^{\lambda
}=t^{-n\left(  \lambda\right)  }$.
\end{remark}

\section{\label{SectHkPro}Hook product formulation}

Recall the definition of the $\left(  q,t\right)  $-hook product%
\[
h_{q,t}\left(  a;\lambda\right)  =\prod\limits_{\left(  i,j\right)  \in
\lambda}\left(  1-aq^{\mathrm{arm}\left(  i,j;\lambda\right)  }t^{l\mathrm{eg}%
\left(  i,j;\lambda\right)  }\right)  ,
\]
where $\mathrm{arm}\left(  i,j;\lambda\right)  =\lambda_{i}-j$ and
$\mathrm{leg}\left(  i,j;\lambda\right)  =\#\left\{  l:i<l\leq\ell\left(
\lambda\right)  ,j\leq\lambda_{l}\right\}  $, where the length of $\lambda$ is
$\ell\left(  \lambda\right)  =\max\left\{  i:\lambda_{i}\geq1\right\}  $. The
terminology refers to the Ferrers diagram of $\lambda$ which consists of boxes
at $\left\{  \left(  i,j\right)  :1\leq i\leq\ell\left(  \lambda\right)
,1\leq j\leq\lambda_{i}\right\}  $.

\begin{proposition}
Suppose $\lambda\in\mathbb{N}_{0}^{N,+}$ and $\ell\left(  \lambda\right)  \leq
L$ for some fixed $L\leq N$ then%
\begin{equation}
\prod\limits_{1\leq i<j\leq L}\frac{\left(  qt^{j-i};q\right)  _{\lambda
_{i}-\lambda_{j}}}{\left(  qt^{j-i+1};q\right)  _{\lambda_{i}-\lambda_{j}}%
}=h_{q,t}\left(  qt;\lambda\right)
%TCIMACRO{\dprod _{i=1}^{L}}%
%BeginExpansion
{\displaystyle\prod_{i=1}^{L}}
%EndExpansion
\left(  qt^{L-i+1};q\right)  _{\lambda_{1}}^{-1}. \label{hookpr1}%
\end{equation}

\end{proposition}

\begin{proof}
The argument is by implicit induction on the last box to be added to the
Ferrers diagram of $\lambda$. Suppose $\lambda_{i}=0$ for $i>k$ and
$\lambda_{k}\geq1$. Define $\lambda^{\prime}$ by $\lambda_{i}^{\prime}%
=\lambda_{i}$ for all $i$ except $\lambda_{k}^{\prime\prime}=\lambda_{k}-1$.
Denote the product on the left side of (\ref{hookpr1}) by $A\left(
\lambda\right)  $, then%
\begin{align*}
\frac{A\left(  \lambda\right)  }{A\left(  \lambda^{\prime}\right)  }  &  =%
%TCIMACRO{\dprod _{i=1}^{k-1}}%
%BeginExpansion
{\displaystyle\prod_{i=1}^{k-1}}
%EndExpansion
\frac{1-q^{\lambda_{i}-\lambda_{k}+1}t^{k-i+1}}{1-q^{\lambda_{i}-\lambda
_{k}+1}t^{k-i}}%
%TCIMACRO{\dprod _{j=k+1}^{L}}%
%BeginExpansion
{\displaystyle\prod_{j=k+1}^{L}}
%EndExpansion
\frac{1-q^{\lambda_{k}}t^{j-k}}{1-q^{\lambda_{k}}t^{i-k+1}}\\
&  =\frac{1-q^{\lambda_{k}}t}{1-q^{\lambda_{k}}t^{L-k+1}}%
%TCIMACRO{\dprod _{i=1}^{k-1}}%
%BeginExpansion
{\displaystyle\prod_{i=1}^{k-1}}
%EndExpansion
\frac{1-q^{\lambda_{i}-\lambda_{k}+1}t^{k-i+1}}{1-q^{\lambda_{i}-\lambda
_{k}+1}t^{k-i}};
\end{align*}
the $j$-product telescopes. Adjoining a box at $\left(  k,\lambda_{k}\right)
$ to the diagram of $\lambda^{\prime}$ causes these changes: $\mathrm{leg}%
\left(  i,\lambda_{k};\lambda\right)  =\mathrm{leg}\left(  i,\lambda
_{k};\lambda^{\prime}\right)  +1$ for $1\leq i<k$, $\mathrm{arm}\left(
k,j;\lambda\right)  =\mathrm{arm}\left(  k,j;\lambda^{\prime}\right)
+1=\lambda_{k}-j$ for $1\leq j<\lambda_{k}$. The calculation also uses
$\mathrm{arm}\left(  i,\lambda_{k};\lambda\right)  =\mathrm{arm}\left(
i,\lambda_{k};\lambda^{\prime}\right)  =\lambda_{i}-\lambda_{k}$;
$\mathrm{leg}\left(  k,j;\lambda^{\prime}\right)  =\mathrm{leg}\left(
k,j;\lambda\right)  =0$. Thus%
\[
\frac{h_{q,t}\left(  qt;\lambda\right)  }{h_{q,t}\left(  qt;\lambda^{\prime
}\right)  }=\left(  1-q^{\lambda_{k}}t\right)
%TCIMACRO{\dprod _{i=1}^{k-1}}%
%BeginExpansion
{\displaystyle\prod_{i=1}^{k-1}}
%EndExpansion
\frac{1-q^{\lambda_{i}-\lambda_{k}+1}t^{k-i+1}}{1-q^{\lambda_{i}-\lambda
_{k}+1}t^{k-i}},
\]
because the change in the product for row $\#k$ is%
\[%
%TCIMACRO{\dprod _{j=1}^{\lambda_{k}}}%
%BeginExpansion
{\displaystyle\prod_{j=1}^{\lambda_{k}}}
%EndExpansion
\left(  1-qtq^{\lambda_{k}-j}\right)
%TCIMACRO{\dprod _{j=1}^{\lambda_{k}-1}}%
%BeginExpansion
{\displaystyle\prod_{j=1}^{\lambda_{k}-1}}
%EndExpansion
\left(  1-qtq^{\lambda_{k}-1-j}\right)  ^{-1}=1-q^{\lambda_{k}}t.
\]
Denote the second product in (\ref{hookpr1}) by $B\left(  \lambda\right)  $
then%
\[
\frac{B\left(  \lambda\right)  }{B\left(  \lambda^{\prime}\right)  }%
=\frac{\left(  qt^{L-k+1};q\right)  _{\lambda_{k}-1}}{\left(  qt^{L-k+1}%
;q\right)  _{\lambda_{k}}}=\frac{1}{1-q^{\lambda_{k}}t^{L-k+1}}.
\]
Hence $\dfrac{A\left(  \lambda\right)  }{A\left(  \lambda^{\prime}\right)
}=\dfrac{h_{q,t}\left(  qt;\lambda\right)  B\left(  \lambda\right)  }%
{h_{q,t}\left(  qt;\lambda^{\prime}\right)  B\left(  \lambda^{\prime}\right)
}.$ To start the induction let $\lambda=\left(  1,0,\ldots,0\right)  $, then
$A\left(  \lambda\right)  =\prod_{j=2}^{L}\dfrac{1-qt^{j-1}}{1-qt^{j}}%
=\dfrac{1-qt}{1-qt^{L}}$, while $h_{q,t}\left(  qt;\lambda\right)  =1-qt$ and
$B\left(  \lambda\right)  =\left(  1-qt^{L}\right)  ^{-1}$. This completes the proof.
\end{proof}

Note that $%
%TCIMACRO{\dprod _{i=1}^{L}}%
%BeginExpansion
{\displaystyle\prod_{i=1}^{L}}
%EndExpansion
\left(  qt^{L-i+1};q\right)  _{\lambda_{1}}=\left(  qt^{L};q,t\right)
_{\lambda}$ (the generalized $\left(  q,t\right)  $-Pochhammer symbol).
Setting $L=N-m-1$ in the Proposition leads to another formulation:

\begin{theorem}
Suppose $\lambda\in\mathcal{N}_{0}^{+}$ then%
\[
V^{\left(  0\right)  }\left(  \lambda\right)  =q^{\beta\left(  \lambda\right)
}t^{e_{0}\left(  \lambda\right)  }\frac{\left(  qt^{N};q,t\right)  _{\lambda
}\left(  qt^{N-m-1};q,t\right)  _{\lambda}}{\left(  qt^{N-1};q,t\right)
_{\lambda}h_{q,t}\left(  qt;\lambda\right)  }.
\]

\end{theorem}

The same method can be applied to $V^{\left(  1\right)  }\left(
\lambda\right)  $ by using $\Xi$ (Definition \ref{defcapxi}).

\begin{theorem}
Suppose $\lambda\in\mathcal{N}_{1}^{+}$ then
\[
V^{\left(  1\right)  }\left(  \lambda\right)  =q^{\beta\left(  \lambda\right)
}t^{e_{1}\left(  \lambda\right)  }\frac{\left(  qt^{-N};q,t^{-1}\right)
_{\lambda}\left(  qt^{-m};q,t^{-1}\right)  _{\lambda}}{\left(  qt^{1-N}%
;q,t^{-1}\right)  _{\lambda}h_{q,1/t}\left(  qt^{-1};\lambda\right)  }.
\]

\end{theorem}

There is a modified definition of leg-length for arbitrary compositions
$\alpha\in\mathbb{N}_{0}^{N}$:%
\[
\mathrm{leg}\left(  i,j;\alpha\right)  =\#\left\{  r:r>i,j\leq\alpha_{r}%
\leq\alpha_{i}\right\}  +\#\left\{  r:r<i,j\leq\alpha_{r}+1\leq\alpha
_{i}\right\}  .
\]
Suppose $\alpha_{i+1}>\alpha_{i}$ then
\begin{equation}
\frac{h_{q,t}\left(  qt,s_{i}\alpha\right)  }{h_{q,t}\left(  qt,\alpha\right)
}=\frac{1-q^{\alpha_{i+1}-\alpha_{i}}t^{r_{\alpha}\left(  i\right)
-r_{\alpha}\left(  i+1\right)  }}{1-tq^{\alpha_{i+1}-\alpha_{i}}t^{r_{\alpha
}\left(  i\right)  -r_{\alpha}\left(  i+1\right)  }}=u_{1}\left(  \frac
{\zeta_{\alpha,E}\left(  i+1\right)  }{\zeta_{\alpha,E}\left(  i\right)
}\right)  ^{-1} \label{hq/hq}%
\end{equation}
from \cite[p.15,Prop. 5]{DL2012} (the argument relates to the box at $\left(
i+1,\alpha_{i}+1\right)  $ in the Ferrers diagram of $\alpha$ and the change
in its leg-length) so that%
\[
h_{q,t}\left(  qt;\alpha^{+}\right)  =\mathcal{R}_{1}\left(  \alpha,E\right)
^{-1}h_{q,t}\left(  qt;\alpha\right)  .
\]
Suppose $\alpha\in\mathcal{N}_{0}\mathbb{\ }$then from $V^{\left(  0\right)
}\left(  \alpha\right)  =\mathcal{R}_{1}\left(  \alpha,E\right)
^{-1}V^{\left(  0\right)  }\left(  \alpha^{+}\right)  $ (see Prop.
\ref{V0alpha}) and (\ref{V0lambda}) we obtain
\[
V^{\left(  0\right)  }\left(  \alpha\right)  =q^{\beta\left(  \alpha\right)
}t^{e_{0}\left(  \alpha^{+}\right)  }\frac{\left(  qt^{N};q,t\right)
_{\alpha^{+}}\left(  qt^{N-m-1};q,t\right)  _{\alpha^{+}}}{\left(
qt^{N-1};q,t\right)  _{\alpha^{+}}h_{q,t}\left(  qt;\alpha^{+}\right)  }.
\]
There is a slight complication for type (1) $V^{\left(  1\right)  }\left(
\alpha\right)  =\left(  -1\right)  ^{\mathrm{inv}\left(  \alpha\right)
}\mathcal{R}_{0}\left(  \alpha,F\right)  ^{-1}V^{\left(  1\right)  }\left(
\alpha^{+}\right)  $%
\[
V^{\left(  1\right)  }\left(  \alpha\right)  =\left(  -t\right)
^{-\mathrm{inv}\left(  \alpha\right)  }V^{\left(  1\right)  }\left(
\alpha^{+}\right)  \frac{h_{q,1/t}\left(  qt^{-1};\alpha^{+}\right)
}{h_{q,1/t}\left(  qt^{-1};\alpha\right)  }.
\]
Start with $\zeta_{\alpha,F}\left(  i\right)  =q^{\alpha_{i}}t^{r_{\alpha
}\left(  i\right)  -1-m}$ for $1\leq i\leq m$ (because $c\left(  i,F\right)
=i-m-1$) and then $\frac{\zeta_{\alpha,F}\left(  i+1\right)  }{\zeta
_{\alpha,F}\left(  i\right)  }=q^{\alpha_{i+1}-\alpha_{i}}t^{r_{\alpha}\left(
i+1\right)  -r_{\alpha}\left(  i\right)  }$. Suppose $\alpha_{i+1}>\alpha_{i}$
and apply $\Xi$ in (\ref{hq/hq}) to obtain%
\[
\frac{h_{q,1/t}\left(  qt^{-1},s_{i}\alpha\right)  }{h_{q,1/t}\left(
qt^{-1},\alpha\right)  }=\frac{1-q^{\alpha_{i+1}-\alpha_{i}}t^{r_{\alpha
}\left(  i+1\right)  -r_{\alpha}\left(  i\right)  }}{1-t^{-1}q^{\alpha
_{i+1}-\alpha_{i}}t^{r_{\alpha}\left(  i+1\right)  -r_{\alpha}\left(
i\right)  }}=tu_{0}\left(  \frac{\zeta_{\alpha,E}\left(  i+1\right)  }%
{\zeta_{\alpha,E}\left(  i\right)  }\right)  ^{-1};
\]
combine with $V^{\left(  1\right)  }\left(  s_{i}\alpha\right)  =-u_{0}\left(
\frac{\zeta_{\alpha,E}\left(  i+1\right)  }{\zeta_{\alpha,E}\left(  i\right)
}\right)  V^{\left(  1\right)  }\left(  \alpha\right)  $ and then%
\[
V^{\left(  1\right)  }\left(  \alpha\right)  h_{q,1/t}\left(  qt^{-1}%
,\alpha\right)  =-t^{-1}V^{\left(  1\right)  }\left(  s_{i}\alpha\right)
h_{q,1/t}\left(  qt^{-1},s_{i}\alpha\right)  .
\]
Thus $V^{\left(  1\right)  }\left(  \alpha\right)  h_{q,1/t}\left(
qt^{-1},\alpha\right)  =\left(  -t\right)  ^{-\mathrm{inv}\left(
\alpha\right)  }V^{\left(  1\right)  }\left(  \alpha^{+}\right)
h_{q,1/t}\left(  qt^{-1},\alpha^{+}\right)  $, and%
\[
V^{\left(  1\right)  }\left(  \alpha\right)  =\left(  -t\right)
^{-\mathrm{inv}\left(  \alpha\right)  }q^{\beta\left(  \alpha\right)
}t^{e_{0}\left(  \alpha^{+}\right)  }\frac{\left(  qt^{-N};q,t^{-1}\right)
_{\alpha^{+}}\left(  qt^{-m};q,t^{-1}\right)  _{\alpha^{+}}}{\left(
qt^{1-N};q,t^{-1}\right)  _{\alpha^{+}}h_{q,1/t}\left(  qt^{-1};\alpha\right)
}.
\]
We have shown that the values of certain Macdonald superpolynomials at special
points $\left(  1,t,\ldots,t^{N-1}\right)  $ or $\left(  1,t^{-1}%
,t^{-2},\ldots,t^{1-N}\right)  $ are products of linear factors of the form
$1-q^{a}t^{b}$ where $a\in\mathbb{N}$ and $-N\leq b\leq-N$.

\section{\label{SectSym}Restricted symmetrization and antisymmetrization}

A type of symmetric Macdonald superpolynomial has been investigated by
Blondeau et al\cite{BDLM2012}. The operators used in their work to define
Macdonald polynomials are significantly different from ours. There are results
on evaluations for these polynomials found by Gonz\'{a}lez and Lapointe
\cite{GL2020}. In this section we consider symmetrization over a subset of the
coordinates, and associated evaluations.

Fix $\lambda\in\mathcal{N}_{0}^{+}$ and consider the sum $p_{\lambda}^{s}%
=\sum\left\{  c_{\beta}M_{\beta,E}:\beta^{+}=\lambda,\beta\in\mathcal{N}%
_{0}\right\}  $ satisfying $\left(  \boldsymbol{T}_{i}-t\right)  p_{\lambda
}^{s}=0$ for $1\leq i<N-m-1$. In this section we determine $p_{\lambda}%
^{s}\left(  x^{\left(  0\right)  };\theta\right)  $. Similarly fix $\lambda
\in\mathcal{N}_{1}^{+}$ and consider the sum $p_{\lambda}^{a}=\sum\left\{
c_{\beta}M_{\beta,F}:\beta^{+}=\lambda,\beta\in\mathcal{N}_{1}\right\}  $
satisfying $\left(  \boldsymbol{T}_{i}+1\right)  p_{\lambda}^{a}=0$ for $1\leq
i<m$, then evaluate $p_{\lambda}^{a}\left(  x^{\left(  1\right)  }%
;\theta\right)  $.

\begin{lemma}
\label{symcofs}Suppose $\beta\in\mathbb{N}_{0}^{N},$ $E^{\prime}\in
\mathcal{Y}_{0}\cup\mathcal{Y}_{1}$ and $\beta_{i}<\beta_{i+1}$ for some $i$.
Let $z=\zeta_{\beta,E^{\prime}}\left(  i+1\right)  /\zeta_{\beta,E^{\prime}%
}\left(  i\right)  $ and let $p=c_{0}M_{s_{i}\beta,E^{\prime}}+c_{1}%
M_{\beta,E^{\prime}}$. If $c_{1}=\dfrac{t-z}{1-z}c_{0}$ then $\left(
\boldsymbol{T}_{i}-t\right)  p=0$ and if $c_{1}=-\dfrac{1-tz}{1-z}c_{0}$ then
$\left(  \boldsymbol{T}_{i}+1\right)  p=0.$
\end{lemma}

\begin{proof}
The general transformation rules are given in matrix form with respect to the
basis $\left[  M_{\beta,E^{\prime}},M_{s_{i}\beta,E^{\prime}}\right]  $%
\[
\boldsymbol{T}_{i}=%
\begin{bmatrix}
-\frac{1-t}{1-z} & \frac{\left(  1-tz\right)  \left(  t-z\right)  }{\left(
1-z\right)  ^{2}}\\
1 & \frac{z\left(  1-t\right)  }{1-z}%
\end{bmatrix}
.
\]
One directly verifies that$,$%
\[
\left(  \boldsymbol{T}_{i}-t\right)
\begin{bmatrix}
\dfrac{t-z}{1-z}\\
1
\end{bmatrix}
=%
\begin{bmatrix}
0\\
0
\end{bmatrix}
,\left(  \boldsymbol{T}_{i}+1\right)
\begin{bmatrix}
-\dfrac{1-tz}{1-z}\\
1
\end{bmatrix}
=%
\begin{bmatrix}
0\\
0
\end{bmatrix}
..
\]

\end{proof}

\begin{definition}
For $\lambda\in\mathcal{N}_{0}^{+}$ set $p_{\lambda}^{s}\left(  x;\theta
\right)  :=\sum\limits_{\alpha\in\mathcal{N}_{0},\alpha^{+}=\lambda
}\mathcal{R}_{0}\left(  \alpha,E\right)  M_{\alpha,E}\left(  x;\theta\right)
$.
\end{definition}

\begin{proposition}
Suppose $\lambda\in\mathcal{N}_{0}^{+}$ then $p_{\lambda}^{s}\left(
x;\theta\right)  $ satisfies $\left(  \boldsymbol{T}_{i}-t\right)  p_{\lambda
}^{\prime}=0$ for $1\leq i<N-m-1$.
\end{proposition}

\begin{proof}
Fix $i$. If $\alpha\in\mathcal{N}_{0},\alpha^{+}=\lambda$ and $\alpha
_{i}=\alpha_{i+1}$ then $\left(  \boldsymbol{T}_{i}-t\right)  M_{\alpha,E}=0$
because $r_{\alpha}\left(  i+1\right)  =r_{\alpha}\left(  i\right)  +1$ and
thus $\zeta_{\alpha,E}\left(  i\right)  =t\zeta_{\alpha,E}\left(  i+1\right)
$. Otherwise take $\alpha_{i}<\alpha_{i+1}$ and $z=\zeta_{\alpha,E}\left(
i+1\right)  /\zeta_{\alpha,E}\left(  i\right)  ,$ then set
\[
p_{\alpha,i}:=\mathcal{R}_{0}\left(  s_{i}\alpha,E\right)  M_{s_{i}\alpha
,E}+\mathcal{R}_{0}\left(  \alpha,E\right)  M_{\alpha,E}=\mathcal{R}%
_{0}\left(  s_{i}\alpha,E\right)  \left\{  M_{s_{i}\alpha,E}+u_{0}\left(
z\right)  M_{\alpha,E}\right\}
\]
by Lemma \ref{RbRb1}, and $u_{0}\left(  z\right)  =\dfrac{t-z}{1-z}$. By Lemma
\ref{symcofs} $\left(  \boldsymbol{T}_{i}-t\right)  p_{a,i}=0$. For each $i$
the sum for $p_{\lambda}^{\prime}$ splits into singletons ($\alpha:\alpha
_{i}=\alpha_{i+1})$ and pairs $\left(  \beta,s_{i}\beta:\beta_{i}<\beta
_{i+1}\right)  $. Each piece is annihilated by $\boldsymbol{T}_{i}-t$.
\end{proof}

There is now enough information on hand to find $p_{\lambda}^{s}\left(
x^{\left(  0\right)  };\theta\right)  ,$since%
\begin{align*}
p_{\lambda}^{s}\left(  x^{\left(  0\right)  };\theta\right)   &
=\sum\limits_{\alpha\in\mathcal{N}_{0},\alpha^{+}=\lambda}\mathcal{R}%
_{0}\left(  \alpha,E\right)  M_{\alpha,E}\left(  x^{\left(  0\right)  }%
;\theta\right) \\
&  =\sum\limits_{\alpha\in\mathcal{N}_{0},\alpha^{+}=\lambda}\frac
{\mathcal{R}_{0}\left(  \alpha,E\right)  }{\mathcal{R}_{1}\left(
\alpha,E\right)  }M_{\lambda,E}\left(  x^{\left(  0\right)  };\theta\right)  .
\end{align*}
This sum can be evaluated using the norm formula established in \cite{D2020}.
This formula applies to arbitrary $\lambda\in\mathbb{N}_{0}^{N,+}$ and
arbitrary sets $E^{\prime}\in\mathcal{Y}_{0}\cup\mathcal{Y}_{1}$, In the
present context which uses only $\alpha\in\mathcal{N}_{0}$ with $\alpha
^{+}=\lambda$ the formula is used with $N$ replaced by $N-m-1$ and the reverse
$\lambda^{-}$ of $\lambda$ is replaced by%
\[
R_{0}\lambda=\left(  \lambda_{N-m-1},\ldots,\lambda_{2},\overset
{N-m-1}{\lambda_{1}},0\ldots,0\right)  .
\]
For $n=0,1,2,\ldots$ define $\left[  n\right]  _{t}!=\prod\limits_{i=1}%
^{n}\left[  i\right]  _{t}$. For $\lambda\in\mathcal{N}_{0}^{+}$ and
$j\leq\lambda_{1}$ let $n_{j}\left(  \lambda\right)  =\#\left\{
l:l<N-m,\lambda_{l}=j\right\}  $. The formula from \cite{D2020} specializes to%
\begin{equation}
\sum\limits_{\alpha\in\mathcal{N}_{0},\alpha^{+}=\lambda}\frac{\mathcal{R}%
_{0}\left(  \alpha,E\right)  }{\mathcal{R}_{1}\left(  \alpha,E\right)  }%
=\frac{\left[  N-m-1\right]  _{t}!}{\prod\limits_{j=0}^{\lambda_{1}}\left[
n_{j}\left(  \lambda\right)  \right]  _{t}!}\frac{1}{\mathcal{R}_{1}\left(
R_{0}\lambda,E\right)  }. \label{RRsum}%
\end{equation}
Note that the multiplier is a type of $t$-multinomial symbol. It is
straightforward to show%
\[
\mathcal{R}_{1}\left(  R_{0}\lambda,E\right)  =\prod\limits_{\substack{1\leq
i<j<N-m\\\lambda_{i}>\lambda_{j}}}\frac{1-q^{\lambda_{i}-\lambda_{j}}%
t^{j-i+1}}{1-q^{\lambda_{i}-\lambda_{j}}t^{j-i}}.
\]
This product can be combined with the $\left(  i,j\right)  $-product in
(\ref{V0lambda}) to show:%
\begin{align*}
p_{\lambda}^{s}\left(  x^{\left(  0\right)  };\theta\right)   &
=q^{\beta\left(  \lambda\right)  }t^{e_{0}\left(  \lambda\right)  }%
\frac{\left[  N-m-1\right]  _{t}!}{\prod\limits_{j=0}^{\lambda_{1}}\left[
n_{j}\left(  \lambda\right)  \right]  _{t}!}\frac{\left(  qt^{N};q,t\right)
_{\lambda}}{\left(  qt^{N-1};q,t\right)  _{\lambda}}\\
&  \times\prod\limits_{\substack{1\leq i<j<N-m\\\lambda_{i}>\lambda_{j}}%
}\frac{\left(  qt^{j-i+1},q\right)  _{\lambda_{i}-\lambda_{j}-1}}{\left(
qt^{j-i},q\right)  _{\lambda_{i}-\lambda_{j}-1}}\tau_{E}\left(  \theta\right)
.
\end{align*}
Also by (\ref{V0alpha})%
\[
p_{\lambda}^{s}\left(  x^{\left(  0\right)  };\theta\right)  =\frac{\left[
N-m-1\right]  _{t}!}{\prod\limits_{j=0}^{\lambda_{1}}\left[  n_{j}\left(
\lambda\right)  \right]  _{t}!}V^{\left(  0\right)  }\left(  R_{0}%
\lambda\right)  \tau_{E}\left(  \theta\right)
\]

\begin{definition}
For $y\in\mathbb{R}^{N-m-1}$ let $y^{\left(  0\right)  }$=$\left(
y_{1},\ldots,y_{N-m-1},t^{N-m-1},\ldots,t^{N-2},t^{N-1}\right)  .$
\end{definition}

Lemma \ref{M0ctauE} applies to each $M_{\alpha,E}$ in the sum for $p_{\lambda
}^{s}\left(  y^{\left(  0\right)  };\theta\right)  $ thus $\left(
T_{i}-t\right)  p_{\lambda}^{s}\left(  y^{\left(  0\right)  };\theta\right)
=0$ for $1\leq i<N-m-1$.

\begin{proposition}
$p_{\lambda}^{s}\left(  y^{\left(  0\right)  };\theta\right)  $ is symmetric
in $y$. In particular $p_{\lambda}^{s}\left(  x^{\left(  0\right)  }%
u;\theta\right)  =p_{\lambda}^{s}\left(  x^{\left(  0\right)  };\theta\right)
$ for any permutation $u$ of $\left\{  1,2,\ldots,N-m-1\right\}  $ (that is
$u\in\mathcal{S}_{N-m-1}\times Id_{m+1}$) .
\end{proposition}

\begin{proof}
Suppose $1\leq i<N-m-1$ then%
\begin{align*}
tp_{\lambda}^{s}\left(  y^{\left(  0\right)  };\theta\right)   &
=\boldsymbol{T}_{i}p_{\lambda}^{s}\left(  y^{\left(  0\right)  };\theta\right)
\\
&  =b\left(  y^{\left(  0\right)  },i\right)  p_{\lambda}^{s}\left(
y^{\left(  0\right)  };\theta\right)  +\left(  T_{i}-b\left(  y^{\left(
0\right)  },i\right)  \right)  p_{\lambda}^{s}\left(  y^{\left(  0\right)
}s_{i};\theta\right) \\
&  =b\left(  y^{\left(  0\right)  },i\right)  p_{\lambda}^{s}\left(
y^{\left(  0\right)  };\theta\right)  +\left(  t-b\left(  y^{\left(  0\right)
},i\right)  \right)  p_{\lambda}^{s}\left(  y^{\left(  0\right)  }s_{i}%
;\theta\right)  ..
\end{align*}%
\[
\left(  t-b\left(  y^{\left(  0\right)  },i\right)  \right)  p_{\lambda}%
^{s}\left(  y^{\left(  0\right)  };\theta\right)  =\left(  t-b\left(
y^{\left(  0\right)  },i\right)  \right)  p_{\lambda}^{s}\left(  y^{\left(
0\right)  }s_{i};\theta\right)  .
\]
The latter is a polynomial identity (after multiplying by $y_{i}-y_{i+1}$) and
thus holds for all $y^{\left(  0\right)  }$, and hence $p_{\lambda}^{s}\left(
y^{\left(  0\right)  }s_{i};\theta\right)  =p_{\lambda}^{s}\left(  y^{\left(
0\right)  };\theta\right)  $.
\end{proof}

Next we consider asymmetric polynomials in type (1). Recall $\alpha
\in\mathcal{N}_{1}$ implies $\zeta_{\alpha,F}\left(  i\right)  =q^{\alpha_{i}%
}t^{r_{\alpha}\left(  i\right)  -m-1}$ if $1\leq i\leq m$ and $\zeta
_{\alpha,F}\left(  i\right)  =t^{N-i}$ if $i>m$.

\begin{definition}
For $\lambda\in\mathcal{N}_{1}^{+}$ set $p_{\lambda}^{a}\left(  x;\theta
\right)  :=\sum\limits_{\alpha\in\mathcal{N}_{1},\alpha^{+}=\lambda}\left(
-1\right)  ^{\mathrm{inv}\left(  \alpha\right)  }\mathcal{R}_{1}\left(
\alpha,F\right)  M_{\alpha,F}\left(  x;\theta\right)  $.
\end{definition}

\begin{proposition}
Suppose $\lambda\in\mathcal{N}_{1}^{+}$ then $p_{\lambda}^{a}\left(
x;\theta\right)  $ satisfies $\left(  \boldsymbol{T}_{i}+t\right)  p_{\lambda
}^{a}=0$ for $1\leq i<m$.
\end{proposition}

\begin{proof}
Fix $i$. If $\alpha\in\mathcal{N}_{1},\alpha^{+}=\lambda$ and $\alpha
_{i}=\alpha_{i+1}$ then $\left(  \boldsymbol{T}_{i}+t\right)  M_{\alpha,F}=0$
because $r_{\alpha}\left(  i+1\right)  =r_{\alpha}\left(  i\right)  +1$ and
thus $\zeta_{\alpha,F}\left(  i\right)  =t^{-1}\zeta_{\alpha,F}\left(
i+1\right)  $. Otherwise take $\alpha_{i}<\alpha_{i+1}$ and $z=\zeta
_{\alpha,E}\left(  i+1\right)  /\zeta_{\alpha,E}\left(  i\right)  ,$ then set
\begin{align*}
p_{\alpha,i}  &  :=\left(  -1\right)  ^{\mathrm{inv}\left(  s_{i}%
\alpha\right)  }\mathcal{R}_{1}\left(  s_{i}\alpha,F\right)  M_{s_{i}\alpha
,F}+\left(  -1\right)  ^{\mathrm{inv}\left(  \alpha\right)  }\mathcal{R}%
_{1}\left(  \alpha,F\right)  M_{\alpha,F}\\
&  =\left(  -1\right)  ^{\mathrm{inv}\left(  s_{i}\alpha\right)  }%
\mathcal{R}_{1}\left(  s_{i}\alpha,F\right)  \left\{  M_{s_{i}\alpha,F}%
-u_{1}\left(  z\right)  M_{\alpha,F}\right\}
\end{align*}
by Lemma \ref{RbRb1}, and $u_{1}\left(  z\right)  =\dfrac{1-tz}{1-z}$. By
Lemma \ref{symcofs} $\left(  \boldsymbol{T}_{i}+1\right)  p_{a,i}=0$. For each
$i$ the sum for $p_{\lambda}^{a}$ splits into singletons ($\alpha:\alpha
_{i}=\alpha_{i+1})$ and pairs $\left(  \beta,s_{i}\beta:\beta_{i}<\beta
_{i+1}\right)  $. Each piece is annihilated by $\boldsymbol{T}_{i}+1$.
\end{proof}

Similarly to the symmetric case we can determine $p_{\lambda}^{a}\left(
x^{\left(  1\right)  };\theta\right)  ,$since (by Proposition \ref{V1alpha})%
\begin{align*}
p_{\lambda}^{a}\left(  x^{\left(  1\right)  };\theta\right)   &
=\sum\limits_{\alpha\in\mathcal{N}_{1},\alpha^{+}=\lambda}\left(  -1\right)
^{\mathrm{inv}\left(  \alpha\right)  }\mathcal{R}_{1}\left(  \alpha,F\right)
M_{\alpha,F}\left(  x^{\left(  1\right)  };\theta\right) \\
&  =\sum\limits_{\alpha\in\mathcal{N}_{1},\alpha^{+}=\lambda}\frac
{\mathcal{R}_{1}\left(  \alpha,F\right)  }{\mathcal{R}_{0}\left(
\alpha,F\right)  }M_{\lambda,F}\left(  x^{\left(  1\right)  };\theta\right)  .
\end{align*}
Formula (\ref{RRsum}) can be adapted to find the sum by applying $\Xi$ and
chasing powers of $t$ (in $\left[  n_{i}\left(  \lambda\right)  \right]
_{t}!$ for example). The typical term in $\frac{\mathcal{R}_{1}\left(
\alpha,F\right)  }{\mathcal{R}_{0}\left(  \alpha,F\right)  }$ is%
\[
\frac{1-tq^{\alpha_{j}-\alpha_{i}}t^{r_{\alpha}\left(  j\right)  -c_{\alpha
}\left(  i\right)  }}{t-q^{\alpha_{j}-\alpha_{i}}t^{r_{\alpha}\left(
j\right)  -c_{\alpha}\left(  i\right)  }}=t^{-1}\frac{1-q^{\alpha_{j}%
-\alpha_{i}}t^{r_{\alpha}\left(  j\right)  -c_{\alpha}\left(  i\right)  +1}%
}{1-q^{\alpha_{j}-\alpha_{i}}t^{r_{\alpha}\left(  j\right)  -c_{\alpha}\left(
i\right)  -1}}%
\]
and applying $\Xi$ yields%
\[
t\frac{1-q^{\alpha_{j}-\alpha_{i}}t^{r_{\alpha}\left(  i\right)  -c_{\alpha
}\left(  j\right)  -1}}{1-q^{\alpha_{j}-\alpha_{i}}t^{r_{\alpha}\left(
i\right)  -c_{\alpha}\left(  j\right)  +1}}=\frac{u_{0}\left(  z\right)
}{u_{1}\left(  z\right)  }%
\]
with $z=q^{\alpha_{j}-\alpha_{i}}t^{r_{\alpha}\left(  i\right)  -c_{\alpha
}\left(  j\right)  }$, the typical term in $\frac{\mathcal{R}_{0}\left(
\alpha,E\right)  }{\mathcal{R}_{1}\left(  \alpha,E\right)  }$ (after the
interchange $m\longleftrightarrow N-m-1$). Thus
\[
\sum\limits_{\alpha\in\mathcal{N}_{1},\alpha^{+}=\lambda}\frac{\mathcal{R}%
_{1}\left(  \alpha,F\right)  }{\mathcal{R}_{0}\left(  \alpha,F\right)  }%
=\Xi\sum\limits_{\alpha^{+}=\lambda}\frac{\mathcal{R}_{0}\left(
\alpha,E\right)  }{\mathcal{R}_{1}\left(  \alpha,E\right)  }.
\]
From $\Xi\left[  n\right]  _{t}=$ $\left[  n\right]  _{1/t}=t^{1-n}\left[
n\right]  _{t}$ and $n_{j}\left(  \lambda\right)  =\#\left\{  l:l\leq
m,\lambda_{l}=j\right\}  $ it follows that
\begin{align*}
\frac{\left[  m\right]  _{1/t}!}{\prod\limits_{j=0}^{\lambda_{1}}\left[
n_{j}\left(  \lambda\right)  \right]  _{1/t}!}  &  =t^{A}\frac{\left[
m\right]  _{t}!}{\prod\limits_{j=0}^{\lambda_{1}}\left[  n_{j}\left(
\lambda\right)  \right]  _{t}!},\\
A  &  =-\frac{m\left(  m-1\right)  }{2}+\sum_{j\geq0}\frac{n_{j}\left(
\lambda\right)  \left(  n_{j}\left(  \lambda\right)  -1\right)  }{2}\\
&  =-\frac{1}{2}m^{2}+\frac{1}{2}m+\frac{1}{2}\sum_{j\geq0}n_{j}\left(
\lambda\right)  ^{2}-\frac{1}{2}m.
\end{align*}
Now let $R_{1}\lambda=\left(  \lambda_{m},\lambda_{m-1},\ldots,\overset
{m}{\lambda_{1}},0\ldots,0\right)  $ and consider
\[
\mathcal{R}_{0}\left(  R_{1}\lambda,F\right)  =\prod\limits_{\substack{1\ell
i<j\leq m\\\lambda_{i}.\lambda_{j}}}\frac{t-q^{\lambda_{i}-\lambda_{j}}%
t^{i-j}}{1-q^{\lambda_{i}-\lambda_{j}t^{i-j}}}=t^{\mathrm{inv}\left(
R_{1}\lambda\right)  }\prod\limits_{\substack{1\ell i<j\leq m\\\lambda
_{i}.\lambda_{j}}}\frac{1-q^{\lambda_{i}-\lambda_{j}}t^{i-j-1}}{1-q^{\lambda
_{i}-\lambda_{j}t^{i-j}}}%
\]
and the transformed%
\begin{align*}
\Xi\mathcal{R}_{1}\left(  R_{0}\lambda,E\right)   &  =\prod
\limits_{\substack{1\ell i<j\leq m\\\lambda_{i}.\lambda_{j}}}\frac
{1-q^{\lambda_{i}-\lambda_{j}}t^{i-j-1}}{1-q^{\lambda_{i}-\lambda_{j}t^{i-j}}%
}\\
&  =t^{-\mathrm{inv}\left(  R_{1}\lambda\right)  }\mathcal{R}_{0}\left(
R_{1}\lambda,F\right)  .
\end{align*}
This results in
\[
\sum\limits_{\alpha\in\mathcal{N}_{1},\alpha^{+}=\lambda}\frac{\mathcal{R}%
_{1}\left(  \alpha,F\right)  }{\mathcal{R}_{0}\left(  \alpha,F\right)
}=t^{A+\mathrm{inv}\left(  R_{1}\lambda\right)  }\frac{\left[  m\right]
_{t}!}{\prod\limits_{j=0}^{\lambda_{1}}\left[  n_{j}\left(  \lambda\right)
\right]  _{t}!}\frac{1}{\mathcal{R}_{0}\left(  R_{1}\lambda,F\right)  }.
\]
We find
\begin{align*}
\mathrm{inv}\left(  R_{1}\lambda\right)   &  =\sum\limits_{1\leq
i<j\leq\lambda_{1}}n_{i}\left(  \lambda\right)  n_{j}\left(  \lambda\right)
=\frac{1}{2}\left\{  \left(  \sum_{i=1}^{\lambda_{1}}n_{i}\left(
\lambda\right)  \right)  ^{2}-\sum_{i=1}^{\lambda_{1}}n_{i}\left(
\lambda\right)  ^{2}\right\} \\
&  =\frac{1}{2}m^{2}-\frac{1}{2}\sum_{i=1}^{\lambda_{1}}n_{i}\left(
\lambda\right)  ^{2}=-A,
\end{align*}
and we have shown%
\begin{align*}
p_{\lambda}^{a}\left(  x^{\left(  1\right)  };\theta\right)   &
=\frac{\left[  m\right]  _{t}!}{\prod\limits_{j=0}^{\lambda_{1}}\left[
n_{j}\left(  \lambda\right)  \right]  _{t}!}\frac{1}{\mathcal{R}_{0}\left(
R_{1}\lambda,F\right)  }M_{\lambda,F}\left(  x^{\left(  1\right)  }%
;\theta\right) \\
&  =\left(  -1\right)  ^{\mathrm{inv}\left(  R_{1}\lambda\right)  }\left[
m\right]  _{t}!\left\{  \prod\limits_{j=0}^{\lambda_{1}}\left[  n_{j}\left(
\lambda\right)  \right]  _{t}!\right\}  ^{-1}V^{\left(  1\right)  }\left(
R_{1}\lambda\right)  \tau_{F}\left(  \theta\right)
\end{align*}
Similarly to type (0) this formula can be further developed:%
\[
\mathcal{R}_{0}\left(  R_{1}\lambda,F\right)  ^{-1}=\prod
\limits_{\substack{1\leq i<j\leq m\\\lambda_{i}>\lambda_{j}}}\frac
{1-q^{\lambda_{i}-\lambda_{j}}t^{i-j}}{t-q^{\lambda_{i}-\lambda_{j}}t^{i-j}%
}=t^{-\mathrm{inv}\left(  R_{1}\lambda\right)  }\prod\limits_{\substack{1\leq
i<j\leq m\\\lambda_{i}>\lambda_{j}}}\frac{1-q^{\lambda_{i}-\lambda_{j}}%
t^{i-j}}{1-q^{\lambda_{i}-\lambda_{j}}t^{i-j-1}}.
\]
Thus (from Theorem \ref{V1lambda})%
\begin{align*}
p_{\lambda}^{a}\left(  x^{\left(  1\right)  };\theta\right)   &
=q^{\beta\left(  \lambda\right)  }t^{A\left(  \lambda\right)  }\frac{\left[
m\right]  _{t}!}{\prod\limits_{j=0}^{\lambda_{1}}\left[  n_{j}\left(
\lambda\right)  \right]  _{t}!}\frac{\left(  qt^{-N};q,t^{-1}\right)
_{\lambda}}{\left(  qt^{1-N};q,t^{-1}\right)  _{\lambda}}\\
&  \times\prod\limits_{\substack{1\leq i<j<N-m\\\lambda_{i}>\lambda_{j}}%
}\frac{\left(  qt^{i-j-1},q\right)  _{\lambda_{i}-\lambda_{j}-1}}{\left(
qt^{i-j},q\right)  _{\lambda_{i}-\lambda_{j}-1}}\tau_{F},\\
A\left(  \lambda\right)   &  =\sum_{i=1}^{m}\lambda_{i}\left(  N-m-i\right)
-\mathrm{inv}\left(  R_{1}\lambda\right)  ,\\
\mathrm{inv}\left(  R_{1}\lambda\right)   &  =\frac{1}{2}m^{2}-\frac{1}{2}%
\sum_{i=1}^{\lambda_{1}}n_{i}\left(  \lambda\right)  ^{2}.
\end{align*}

\begin{definition}
For $y\in\mathbb{R}^{m}$ let $y^{\left(  1\right)  }$=$\left(  y_{1}%
,\ldots,y_{m},t^{-m},\ldots,t^{2-N},t^{1-N}\right)  $
\end{definition}

Lemma \ref{M1ctauF} applies to each $M_{\alpha,F}$ in the sum for $p_{\lambda
}^{a}\left(  y^{\left(  1\right)  };\theta\right)  $ thus $\left(
T_{i}+1\right)  p_{\lambda}^{a}\left(  y^{\left(  1\right)  };\theta\right)
=0$ for $1\leq i<m$

\begin{proposition}
$p_{\lambda}^{a}\left(  y^{\left(  1\right)  };\theta\right)  $ is symmetric
in $y$. In particular $p_{\lambda}^{a}\left(  x^{\left(  1\right)  }%
u;\theta\right)  =p_{\lambda}^{a}\left(  x^{\left(  1\right)  };\theta\right)
$ for any permutation $u$ of $\left\{  1,2,\ldots,m\right\}  $ (that is
$u\in\mathcal{S}_{m}\times Id_{N-m}$) .
\end{proposition}

\begin{proof}
Suppose $1\leq i<m$ then%
\begin{align*}
-p_{\lambda}^{a}\left(  y^{\left(  1\right)  };\theta\right)   &
=\boldsymbol{T}_{i}p_{\lambda}^{a}\left(  y^{\left(  1\right)  };\theta\right)
\\
&  =b\left(  y^{\left(  1\right)  },i\right)  p_{\lambda}^{sa}\left(
y^{\left(  1\right)  };\theta\right)  +\left(  T_{i}-b\left(  y^{\left(
1\right)  },i\right)  \right)  p_{\lambda}^{a}\left(  y^{\left(  1\right)
}s_{i};\theta\right) \\
&  =b\left(  y^{\left(  1\right)  },i\right)  p_{\lambda}^{a}\left(
y^{\left(  1\right)  };\theta\right)  -\left(  1+b\left(  y^{\left(  1\right)
},i\right)  \right)  p_{\lambda}^{a}\left(  y^{\left(  1\right)  }s_{i}%
;\theta\right) \\
\left(  1+b\left(  y^{\left(  1\right)  },i\right)  \right)  p_{\lambda}%
^{a}\left(  y^{\left(  1\right)  };\theta\right)   &  =\left(  1+b\left(
y^{\left(  1\right)  },i\right)  \right)  p_{\lambda}^{a}\left(  y^{\left(
1\right)  }s_{i};\theta\right)  .
\end{align*}
The latter is a polynomial identity (after multiplying by $y_{i}-y_{i+1}$) and
thus holds for all $y^{\left(  1\right)  }$, and hence $p_{\lambda}%
^{as}\left(  y^{\left(  1\right)  }s_{i};\theta\right)  =p_{\lambda}%
^{a}\left(  y^{\left(  1\right)  }s_{i};\theta\right)  $.
\end{proof}

\section{\label{SectConc}Conclusion and Future Directions}

We established formulas for evaluating a relatively restrictive class of
nonsymmetric polynomials at special points. The values have product form,
whose typical terms are $1-q^{a}t^{b}$ where $a,b\in\mathbb{Z}$ and
$\left\vert b\right\vert \leq N$. The labels $E^{\prime}$ of the Macdonald
polynomials $M_{\alpha,E^{\prime}}$ have only two possibilities out of many,
$\binom{N-1}{m}$ for the isotype $\left(  N-m,1^{m}\right)  $. Computational
experiments suggest there are no other evaluations with this simple form.
Perhaps there are formulas combining sums and products, but we have no
conjectures to offer. However there are other evaluations to be studied: these
relate to singular polynomials. This refers to the situation where the
parameters $q,t$ satisfy a relation like $q^{a}t^{b}=1$ and a polynomial
$M_{\alpha,E^{\prime}}$ satisfies $\xi_{i}M_{\alpha,E^{\prime}}%
=\boldsymbol{\omega}_{i}M_{\alpha,E^{\prime}}$ for $1\leq i\leq N$ The
Jucys-Murphy operators on $s\mathcal{P}_{m}$ are defined in terms of $\left\{
\boldsymbol{T}_{i}\right\}  $ (see (\ref{defTTp})): $\boldsymbol{\omega}%
_{N}=1,\boldsymbol{\omega}_{i}=t^{-1}\boldsymbol{T}_{i}\boldsymbol{\omega
}_{i+1}\boldsymbol{T}_{i}$ for $1\leq i<N$. Of course finding these singular
parameters $\left(  q,t\right)  $ is already a research problem by itself. For
small $N$ and degree we can find some examples (with computer algebra) and
test evaluations. It appears there are interesting results to find.

Consider $N=6$ and $\mathcal{P}_{1,0},\alpha=(1,1,0,0,0,0)$ (of isotype
$\left(  5,1\right)  $). The spectral vector of $M_{\alpha,\left\{
5,6\right\}  }$ is $\left(  qt^{4},qt^{3},t^{2},t,t^{=1},1\right)  .$Let
\[
x=\left(  x_{1},x_{2},t^{2},t,t^{-1},1\right)  .
\]
The polynomial $M_{\alpha,\left\{  5,6\right\}  }$ is singular for $qt^{3}=1$
and
\begin{align*}
M_{\alpha,\left\{  5,6\right\}  }\left(  x;\theta\right)   &  =t^{16}\left(
x_{1}-1\right)  \left(  x_{2}-1\right)  \left(  t^{4}\theta_{6}-t^{5}%
\theta_{5}\right)  ,\\
\tau_{\left\{  5,6\right\}  }  &  =t^{4}\theta_{6}-t^{5}\theta_{5}.
\end{align*}
Similarly $M_{\alpha,\left\{  4,6\right\}  }$ is singular at $qt^{3}=1$; its
spectral vector is $\left(  qt^{4},qt^{3},t^{2},t^{-1},t,1\right)  $ and for
$x^{\prime}=\left(  x_{1},x_{2},t^{2},t^{-1},t,1\right)  $%
\begin{align*}
M_{\alpha,\left\{  4,6\right\}  }\left(  x^{\prime};\theta\right)   &
=t^{16}\left(  x_{1}-1\right)  \left(  x_{2}-1\right)  \tau_{\left\{
4,6\right\}  }\left(  \theta\right)  ,\\
\tau_{\left\{  4,6\right\}  }  &  =-t^{6}\theta_{4}+\frac{t^{5}}{1+t}\left(
\theta_{5}+\theta_{6}\right)  .
\end{align*}
As well $M_{\alpha,\left\{  3,6\right\}  }$ is singular at $qt^{3}=1$; its
spectral vector is $\left(  qt^{4},qt^{3},t^{-1},t^{2},t,1\right)  $ and for
$x^{\prime\prime}=\left(  x_{1},x_{2},t^{-1},t^{2},t,1\right)  $%
\begin{align*}
M_{\alpha,\left\{  3,6\right\}  }\left(  x^{\prime\prime};\theta\right)   &
=t^{16}\left(  x_{1}-1\right)  \left(  x_{2}-1\right)  \tau_{\left\{
3,6\right\}  }\left(  \theta\right)  ,\\
\tau_{\left\{  3,6\right\}  }  &  =-t^{7}\theta_{3}+\frac{t^{6}}{1+t+t^{2}%
}\left(  \theta_{4}+\theta_{5}+\theta_{6}\right)  .
\end{align*}
For an example with higher degree consider $N=6,\mathcal{P}_{4,1}$ and
$\beta=\left(  2,1,0,0,0,0\right)  ,F=\left\{  1,2,3\right\}  $ (the isotype
is $\left(  3,1^{3}\right)  $) Then $\zeta_{\beta,F}=\left(  q^{2}%
t^{-3},qt^{-2},t^{-1},t^{2},t,1\right)  $ and $M_{\beta,F}$ is singular for
$q=t^{2}.$ At $x=\left(  x_{1},x_{2},t^{-1},t^{2},t,1\right)  $ we find%
\begin{align*}
M_{\beta,F}\left(  x;\theta\right)   &  =t^{8}\left(  x_{1}-tx_{2}\right)
\left(  x_{1}-1\right)  \left(  x_{2}-1\right)  \tau_{F},\\
\tau_{F}  &  =-\theta_{1}\theta_{2}\theta_{3}\left(  \theta_{4}+\theta
_{5}+\theta_{6}\right)  .
\end{align*}

We would expect an evaluation formula involving the elements of the spectral
vector for which $\alpha_{i}=0$ and with as many free variables as nonzero
elements of $\alpha$. There is a nice necessary condition for a singular
value: the $t$-exponents of the specialized spectral vector have to agree with
the content vector of an RSYT. For example set $q=t^{2}$ in $\zeta_{\beta,F}$
with the result $\left(  t,1,t^{-1},t^{2},t,1\right)  $, and $\left[
1,0,-1,2,1,0\right]  $ is the content vector of%
\[%
\begin{bmatrix}
6 & 5 & 4\\
3 & 2 & 1
\end{bmatrix}
.
\]
Then $M_{\gamma,F}$ with $\gamma=\left(  2,1,0,0,0,0\right)  $ can not be
singular at $q=t^{2}$: the spectral vector $\zeta_{\gamma,F}=\left(
1,t,t^{-1},t^{2},t,1\right)  $ and $\left[  0,1,-1,2,1,0\right]  $ is not the
content vector of any RSYT.

There are interesting results dealing with singular Macdonald superpolynomials
waiting to be found.

\end{document}